\documentclass[11pt]{article}

\usepackage{amsmath,amssymb,amsfonts,textcomp,amsthm,xifthen,psfrag,graphicx,color,MnSymbol}

\usepackage{amsmath,amssymb,amsfonts,textcomp,xifthen,psfrag,graphicx,color,MnSymbol}
\usepackage[T1]{fontenc}

\usepackage{hyperref}

\usepackage{pgfplots}
\pgfplotsset{compat=newest}
\usepgfplotslibrary{groupplots}
\usepgfplotslibrary{external}
\date{}

\newtheorem{theorem}{Theorem}
\newtheorem{lemma}[theorem]{Lemma}
\newtheorem{cor}[theorem]{Corollary}
\newtheorem{prop}[theorem]{Proposition}
\newtheorem{remark}[theorem]{Remark}
\newtheorem{ass}[theorem]{Assumption}
\theoremstyle{definition} 

\newcommand{\<}{\langle{}}
\renewcommand{\>}{\rangle}

\newcommand{\ip}[2]{\llangle#1\hspace*{.5mm},#2\rrangle}
\newcommand{\dual}[2]{\<#1\hspace*{.5mm},#2\>}

\newcommand{\vdual}[2]{(#1\hspace*{.5mm},#2)}

\newcommand{\bc}[1]{\mathrm{b\!c}(#1)}
\newcommand{\G}[1]{{\Gamma_{\mathrm{#1}}}}
\newcommand{\barG}[1]{\bar\Gamma_{\mathrm{#1}}}

\newcommand{\diam}{\mathrm{diam}}

\newcommand{\wat}{\widehat}
\newcommand{\transp}{\mathsf{T}}

\newcommand{\Grad}{\boldsymbol{\nabla}}
\newcommand{\sGrad}{\boldsymbol{\varepsilon}}

\def\Div{{\rm\bf div\,}}

\def\grad{\nabla}

\def\GG{\boldsymbol{G}}
\def\MM{\boldsymbol{M}}
\def\SS{\boldsymbol{S}}

\def\NN{\boldsymbol{N}}

\def\II{\boldsymbol{I}}

\def\bq{\boldsymbol{q}}

\newcommand{\bL}{\ensuremath{\boldsymbol{L}}}
\newcommand{\LL}{\ensuremath{\mathbb{L}}}

\def\tM{\wat{M\!p}}
\def\tpsi{\wat{\psi\!\eta}}
\def\tphi{\wat{\phi\!\xi}}
\def\tN{\wat{N\!q}}

\newcommand{\bg}{\boldsymbol{g}}

\newcommand{\uu}{\mathfrak{u}}

\newcommand{\deltauu}{\delta\!\uu}
\newcommand{\vv}{\mathfrak{v}}

\newcommand{\VV}{\ensuremath{\mathfrak{V}}}

\newcommand{\brho}{\ensuremath{\boldsymbol{\rho}}}
\newcommand{\bpsi}{\ensuremath{\boldsymbol{\psi}}}
\newcommand{\bphi}{\ensuremath{\boldsymbol{\phi}}}

\newcommand{\bxi}{\ensuremath{\boldsymbol{\xi}}}
\newcommand{\bchi}{\ensuremath{\boldsymbol{\chi}}}

\newcommand{\deltar}{\delta\!r}
\newcommand{\deltau}{\delta\!u}

\newcommand{\deltachi}{\boldsymbol{\delta}\!\bchi}

\newcommand{\deltav}{\delta\!v}

\newcommand{\tracepsi}[1]{\mathrm{tr}_{#1}^\psi}
\newcommand{\traceM}[1]{\mathrm{tr}_{#1}^M}

\newcommand{\Hdiv}[1]{{\bH(\div\!,#1)}}
\newcommand{\Hrot}[1]{{\bH(\rot,#1)}}
\newcommand{\Hrotz}[1]{{\bH_0(\rot,#1)}}
\newcommand{\HDiv}[1]{{\HH(\Div\!,#1)}}
\newcommand{\HDivs}[1]{{\HH^s(\Div\!,#1)}}

\newcommand{\Hpsiz}{\ensuremath{\boldsymbol{H}^{\mathrm{\psi}}_0}}
\newcommand{\Hpsih}{\ensuremath{\boldsymbol{H}^{\mathrm{\psi}}_h}}
\newcommand{\HMz}{\ensuremath{\boldsymbol{H}^{\mathrm{M}}_0}}
\newcommand{\HMh}{\ensuremath{\boldsymbol{H}^{\mathrm{M}}_h}}

\newcommand{\HH}{\ensuremath{\boldsymbol{H}}}
\newcommand{\bH}{\ensuremath{\boldsymbol{H}}}

\DeclareMathOperator{\curl}{\mathbf{curl}}
\DeclareMathOperator{\rot}{rot}

\def\div{{\rm div\,}}

\newcommand{\ttt}{{\mathfrak{T}}}

\newcommand{\R}{\ensuremath{\mathbb{R}}}

\newcommand{\nn}{\ensuremath{\boldsymbol{n}}}

\newcommand{\cC}{\ensuremath{\mathcal{C}}}
\newcommand{\cCinv}{\ensuremath{\mathcal{C}^{-1}}}
\newcommand{\cT}{\ensuremath{\mathcal{T}}}

\newcommand{\cS}{\ensuremath{\mathcal{S}}}

\newcommand{\bt}{\ensuremath{\boldsymbol{t}}}

\newcommand{\eeta}{{\boldsymbol\eta}}

\title{A DPG method for Reissner--Mindlin plates
\thanks{Supported by ANID (formerly CONICYT) through FONDECYT projects 1190009, 1210391}}
\author{
Thomas~F\"uhrer$^\dagger$
\and
Norbert Heuer\thanks{
Facultad de Matem\'aticas, Pontificia Universidad Cat\'olica de Chile,
Avenida Vicu\~na Mackenna 4860, Santiago, Chile,
email: {\tt \{tofuhrer,nheuer\}@mat.uc.cl}}
\and
Antti H. Niemi\thanks{
Civil Engineering Research Unit, Faculty of Technology, University of Oulu,
Erkki Koiso-Kanttilan katu 5, 90570 Oulu, Finland,
email: {\tt antti.niemi@oulu.fi}}}

\begin{document}
\maketitle
\begin{abstract}
We present a discontinuous Petrov--Galerkin (DPG) method with optimal test functions
for the Reissner--Mindlin plate bending model. Our method is based on a variational
formulation that utilizes a Helmholtz decomposition of the shear force.
It produces approximations of the primitive variables and the bending moments.
For any canonical selection of boundary conditions the method converges quasi-optimally.
In the case of hard-clamped convex plates, we prove that the lowest-order scheme is locking free.
Several numerical experiments confirm our results.

\end{abstract}

\section{Introduction}

We present and analyze a discontinuous Petrov--Galerkin method with optimal test functions
(DPG method) for the Reissner--Mindlin plate bending model.
The model considers the influence of transverse shear deformations in the strain-displacement relations
and may be viewed as a correction of the classical Kirchhoff--Love model.
Most prominently, the effective shear forces and concentrated corner forces,
that some scholars consider to be unnatural, are not needed in the Reissner--Mindlin model
to reduce the number of independent edge reactions. While for thin isotropic plates the results of
these models do not differ much, the more accurate kinematic assumptions of the Reissner--Mindlin
model lead also to practical improvement of results for thick and orthotropic plates.
On the other hand, principal advantages of the DPG framework are that it provides automatically
stable approximations for any conforming discretization space \cite{DemkowiczG_11_CDP,DemkowiczG_14_ODM},
and that adaptivity is a built-in feature, cf.~\cite{DemkowiczGN_12_CDP,CarstensenDG_14_PEC}.
For mechanical models of thin structures, the framework can be used to construct formulations
with the most relevant stress and displacement quantities of interest as primal variables. 
The underlying assumption is to have appropriate well-posed variational formulations
with product test spaces. The construction of such formulations and suitable approximation spaces
for the Reissner--Mindlin model is by no means trivial.

It should be noted that there is a long history of developing various other kinds of discretizations
for the Reissner--Mindlin model, and there are several strategies to tackle the numerical
phenomenon of transverse shear locking. Among them are, combined with mixed formulations,
Helmholtz decompositions of the shear force \cite{BrezziF_86_NAM,ArnoldF_89_UAF,BrambleS_98_NNL},
non-conforming approximations \cite{ArnoldF_89_UAF}, and reduced integration
\cite{BrezziBF_89_MIE,BrezziFS_91_EAM,ArnoldF_97_ALL}, to name a few classical results.
Some more recent approaches include special elements \cite{ArnoldBFM_07_LFR},
weakly over-penalized discontinuous Galerkin methods \cite{BoesingC_15_DGW,BoesingC_15_WOP}
and the DDR (discrete de Rham) complex method \cite{DiPietroD_DMR}.
DPG schemes allow for very general meshes, see~\cite{BacutaDMX_21_ANC} for a non-conforming setup,
similarly to virtual elements which have been developed for the Reissner--Mindlin model,
see \cite{BeiraodaVeigaMR_19_VES}. We also note that DPG schemes, being of minimum
residual type, are related with least squares approaches. Such methods have been
studied for a perturbed Stokes problem which behaves like the Reissner--Mindlin model,
see~\cite{CaiYZ_99_LSF,Cai_00_LSP}.
This list is far from being complete. For a more detailed discussion we refer to the recent
paper \cite{DiPietroD_DMR} on the DDR method.

In the present paper, we continue our development of DPG schemes for plate and shell structures
that started with the analysis of the Kirchhoff--Love plate model in \cite{FuehrerHN_19_UFK}.
Subsequently, we extended our techniques to the Reissner--Mindlin
model \cite{FuehrerHS_20_UFR} and shallow shells of Koiter type \cite{FuehrerHN_DMS}.
In both problems, Reissner--Mindlin plates and Koiter shells, different kinds of locking phenomena
appear and it must be stressed that DPG schemes are not automatically locking free.
In \cite{FuehrerHN_DMS} we dealt with the membrane locking by increasing the approximation order
of a trace variable. On the other hand, our previous Reissner--Mindlin plate study
\cite{FuehrerHS_20_UFR} focused on developing a variational formulation that converges to the
Kirchhoff--Love case \cite{FuehrerHN_19_UFK} when the plate thickness tends to zero.
Whether that approach is locking free depends on the construction of discrete trace spaces,
not considered in \cite{FuehrerHS_20_UFR}. Let us also mention the study \cite{CaloCN_14_ADP}
of a DPG scheme for Reissner--Mindlin plates, though their analysis considers the plate thickness
to be fixed.

Here, we complete the picture by utilizing a Helmholtz decomposition of the shear force
(after suitable scaling), as proposed by Brezzi and Fortin \cite{BrezziF_86_NAM} and thoroughly
analyzed by Arnold and Falk in \cite{ArnoldF_89_UAF}. In this way, a variational formulation
consists of three stages, the first to determine the irrotational component $\grad r$ of the
scaled shear force $\bq$, the second to determine the solenoidal component $\curl p$ of $\bq$
together with other variables, and the third to recover the vertical deflection $u$.
Stages one and three are simple Poisson problems which can be solved by standard finite elements,
whereas stage two reflects the very Reissner--Mindlin model and is solved by a DPG scheme.
Apart from variable $p$, it delivers the rotation field, $\bpsi$, and the bending moments, $\MM$,
along with the artificial variable $\eeta:=t\curl p$ where $t$ is the plate thickness.

Let us recall two important issues that determine the behavior of DPG approximations.
First, optimal test functions are a theoretical construct and have to be discretized
in practice. In order to maintain discrete stability, the existence of Fortin operators
has to be shown, cf.~\cite{GopalakrishnanQ_14_APD}. For our Kirchhoff--Love setting
from \cite{FuehrerHN_19_UFK}, this is done in \cite{FuehrerH_19_FDD}.
In the case of the Reissner--Mindlin model, the existence of Fortin operators remains open.
Second, in cases where the stability analysis requires Poincar\'e-type estimates,
norms in test spaces have to be properly scaled in order for DPG schemes to be
robust for larger domains. Otherwise approximations suffer
from long pre-asymptotic ranges of reduced convergence, a type of locking phenomenon.
For a detailed analysis we refer to \cite{FuehrerH_21_RDM}, and we note that
a scaling of norms is required for least-squares methods as well,
see~\cite[Section~3]{BringmannC_17_hAL}. In order to not complicate the presentation,
we restrain from providing all the details that are required to have a domain-robust approximation.
Instead, we collect all the needed changes in \S\ref{sec_large} and state the corresponding
error estimates without proof.

An overview of the remainder is as follows. In the next section we specify the model problem,
prescribe allowed boundary conditions, define some basic spaces, and present the
Helmholtz decomposition. Section~\ref{sec_traces} gives all the definitions and technical
results on spaces, norms, and trace operators. 
In Section~\ref{sec_DPG} we make use of the trace operators and spaces to derive our ultra-weak
formulation. We list the three variational stages and claim their stability
in Theorem~\ref{thm_stab} and Corollary~\ref{cor_stab}. Subsequently we present a discretization
of the three stages and state their $t$-robust quasi-optimal convergence in Theorem~\ref{thm_DPG}.
In the remainder of Section~\ref{sec_DPG} we provide proofs of the stated results.
In order to prove that our scheme is locking free, certain regularity results are needed.
They usually require convexity of the domain and hard-clamped boundary conditions.
In Section~\ref{sec_lf} we specify these assumptions, and recall from
\cite{ArnoldF_89_UAF} that they are satisfied for a special case.
We then state and prove a locking-free a priori error estimate
for a lowest-order scheme that uses canonical bases.
Finally, in Section~\ref{sec_num} we present several numerical experiments that illustrate
our theoretical results. In particular, they confirm that our scheme is locking free for
convex hard-clamped plates. It exhibits locking-free approximations also for other boundary
conditions and an example on a non-convex polygonal plate.

Throughout the paper, notation $a\lesssim b$ means that $a\le cb$ with an unspecified
generic constant $c>0$ that is independent of the (scaled) plate thickness $t$ and the
underlying mesh $\cT$. Notation $a\gtrsim b$ means that $b\lesssim a$,
and $a\simeq b$ indicates $a\lesssim b$ and $a\gtrsim b$.


\section{Model problem} \label{sec_model}

Let $\Omega\subset\R^2$ be a bounded simply connected Lipschitz domain
with boundary $\Gamma=\partial\Omega$. More specifically, for simplicity of the discrete analysis,
we assume that $\Omega$ is a polygon.
We are considering the Reissner--Mindlin plate bending model
with linearly elastic, homogeneous and isotropic material, described by the constitutive relations
\begin{align}
   \nonumber
   \bq &= \kappa Gt(\grad u-\bpsi),\\
   \label{RM1b}
   \MM &= -Dt^3[\nu\mathrm{tr}(\sGrad{\bpsi})\II+(1-\nu)\sGrad{\bpsi}]
\end{align}
and the equilibrium equations
\begin{align*}
   -\div\bq&=f,\\ \bq&=\Div\MM
\end{align*}
on $\Omega$. Here, $\Omega$ is the mid-surface of the plate with thickness $t>0$, $f$ the transversal
bending load, $u$ the transverse deflection, $\bpsi$ the rotation vector, $\bq$ the shear force vector,
$\MM$ the bending moment tensor, $\II$ the identity tensor,
and $\sGrad$ the symmetric gradient, $\sGrad{\bpsi}:=\frac 12(\grad\bpsi+(\grad\bpsi)^\transp)$.
Furthermore, $\nu\in (-1,1/2]$ is the Poisson ratio, $\kappa>0$ the shear correction factor, and
\[
   G=\frac E{2(1+\nu)},\quad D=\frac E{12(1-\nu^2)}
\]
with the Young modulus $E>0$. The operator $\div$ is the standard divergence, and
$\Div$ is the row-wise divergence when writing second-order tensors as $2\times 2$ matrix functions.

Relation \eqref{RM1b} between $\MM$ and $\bpsi$ can be written like
\[
   \MM=-t^3\cC\sGrad{\bpsi}
\]
where $\cC$ is the positive definite plane stress constitutive tensor that is independent of $t$.
It is an isomorphism within the space $\LL_2^s(\Omega)$ of symmetric $L_2(\Omega)$-tensors.
Since the dependence of the problem on $\kappa G$ is not critical, we simply select $\kappa G=1$.
Then, re-scaling $f\to t^3 f$, $\MM\to t^3\MM$ and $\bq\to t^3\bq$,
the Reissner--Mindlin model simplifies to the system
\begin{align} \label{RM3}
   -\div\bq = f,\quad
   \MM+\cC\sGrad{\bpsi} = 0,\quad
   \Div\MM-\bq =0,\quad
   \bq = t^{-2}(\grad u - \bpsi).
\end{align}
Canonical (homogeneous) boundary conditions are given on mutually exclusive subsets
$\G{hc}$, $\G{sc}$, $\G{hss}$, $\G{sss}$, $\G{f}$ of $\Gamma$ (for simplicity we
assume that they are connected; they can be empty and the closure of their union equals $\Gamma$),
\begin{alignat*}{2}
   \text{hard clamped (hc):}\quad & \bpsi=0,\ u=0
      &&\text{on}\ \G{hc},\\
   \text{soft clamped (sc):}\quad & \bpsi\cdot\nn = \bt\cdot\MM\nn = u =0 \qquad
      &&\text{on}\ \G{sc},\\
   \text{hard simple support (hss):}\quad & \nn\cdot\MM\nn = \bpsi\cdot\bt = u = 0
      &&\text{on}\ \G{hss},\\
   \text{soft simple support (sss):}\quad & \MM\nn = 0,\ u = 0
      &&\text{on}\ \G{sss},\\
   \text{free (f):}\quad                  & \MM\nn = 0,\ \bq\cdot\nn = 0
      &&\text{on}\ \G{f}.
\end{alignat*}
Here, $\nn$ and $\bt$ are the unit exterior normal and tangential vectors along $\Gamma$, respectively.

In the following we use the notation $\G{u}:=\Gamma\setminus\barG{f}$ for the part of the boundary
with imposed zero deflection.
Throughout, our assumptions are that $\G{u}$ has positive measure and that the boundary condition
for rotation $\bpsi$ (referred to as ``$\bc{\psi}$'') eliminates ``rigid rotations''
so that the Korn inequality $\|\bpsi\| \lesssim \|\sGrad\bpsi\|$ holds.
Here, $\|\cdot\|$ denotes the $L_2(\Omega)$-norm, generically for scalar, vector and tensor-valued
functions. In the case of $\G{f}=\emptyset$ we will need the quotient space
$L_2(\Omega)/\R$ with quotient norm. Also, some scaling will be different for different boundary
conditions. For brevity we will write
\begin{equation} \label{star}
\begin{aligned}
   &\text{if}\ \G{f}=\emptyset:
   && L_2^*(\Omega):=L_2(\Omega)/\R,\quad && H^1_*(\Omega) := H^1(\Omega)/\R,\quad
   && \|\cdot\|_*:=\inf_{c\in\R}\|\cdot-c\|,\\
   &\text{if}\ \G{f}\not=\emptyset:
   && L_2^*(\Omega):=L_2(\Omega),\quad    && H^1_*(\Omega) := H^1(\Omega),\quad
   && \|\cdot\|_*:=\|\cdot\|,\\
   &\text{if}\ \G{sc}=\G{sss}=\emptyset:\quad && t_*:=1,\\
   &\text{if}\ \G{sc}\cup\G{sss}\not=\emptyset: && t_*:=t.
\end{aligned}
\end{equation}
Here, $H^1(\Omega)$ is the standard Sobolev space.
We note that parameter $t_*$ will appear as a weighting factor in the test norm and determines the
weighting $t_*^{-1}$ of one of the unknown trace terms. Now, in case of boundary conditions
implying $\G{sc}=\G{sss}=\emptyset$, assignment $t_*=1$ implies that weight $t_*^{-1}$ on the ansatz
side does not blow up when $t\to 0$. This will be important when proving that our scheme is
locking free. Whether the selection of $t_*=1$ is possible will become clear in the stability
analysis of the adjoint problem, cf.~\S\ref{sec_adj} below.

In order to formulate the needed Korn and Poincar\'e inequalities we introduce some
spaces with boundary conditions.
We denote $\bH^1(\Omega):=\bigl(H^1(\Omega)\bigr)^2$ and define
\begin{subequations} \label{H1_psi_p}
\begin{align}
   \bH^1_\psi(\Omega) &:= \{\bchi\in\bH^1(\Omega);\; \bchi\ \text{satisfies}\ \bc{\psi}\},\\
   H^1_u(\Omega) &:= \{v\in H^1(\Omega);\; v=0 \text{ on } \Gamma\setminus\barG{f}\},\\
   H^1_p(\Omega) &:= \{v\in H^1_*(\Omega);\; v|_\G{f}=0\}.
\end{align}
\end{subequations}
Notation $H^1_p(\Omega)$ stems from a variable $p$, later introduced through a Helmholtz
decomposition, with boundary condition complementary to that of $u$.
Then we will use the Korn and Poincar\'e inequalities
\begin{align} \label{Korn}
   \|\bchi\| + \|\Grad\bchi\| \lesssim \|\sGrad\bchi\| \ \forall\bchi\in \bH^1_\psi(\Omega),
\end{align}
\begin{align} \label{Poincare}
   \|v\|_* \lesssim \|\grad v\| \ \forall v\in H^1_p(\Omega).
\end{align}

\begin{remark} \label{rem_d}
Inequalities \eqref{Korn}, \eqref{Poincare} are valid for a fixed domain. The hidden constants
depend on the size of $\Omega$.
Since these bounds influence the selection of test norms for our DPG scheme,
it is critical to include scaling parameters when considering larger, even moderately sized,
domains. We refer to \cite{FuehrerH_21_RDM} for a detailed discussion.
For ease of presentation, our analysis does not consider such a scaling, i.e., we assume
$\diam(\Omega)$ to be the length unit. In Section~\ref{sec_large}, we specify
the changes that are required in order to have a DPG scheme that is robust with respect
to the diameter of $\Omega$.
\end{remark}

\subsection{Helmholtz decomposition}

Let us introduce the operators
\begin{align*}
   &\curl z:=(\partial_y z,-\partial_x z)^T,\qquad\
   \rot \bq := \partial_x q_y - \partial_y q_x\quad\text{for}\ \bq=(q_x,q_y)^T
   \qquad\quad\text{(formal adjoints)}.
\end{align*}
Following \cite{BrezziF_86_NAM,ArnoldF_89_UAF}, we use a Helmholtz decomposition
of $\bq=t^{-2}(\grad u-\bpsi)$. Defining
\begin{align} \label{prob_r}
   r\in H^1_u(\Omega):\quad
   \vdual{\grad r}{\grad\deltar} = \vdual{f}{\deltar}
   \quad\forall\deltar\in H^1_u(\Omega)
\end{align}
(with $L_2(\Omega)$-duality $\vdual{\cdot}{\cdot}$)
it follows that $\div(\bq-\grad r)=0$ so that
\begin{align} \label{dec}
   \bq = t^{-2}(\grad u-\bpsi) = \grad r + \curl p
\end{align}
with $p\in H^1(\Omega)$ satisfying
\begin{align} \label{prob_p}
   \rot\bigl(t^2\curl p+\bpsi) = 0\quad\text{in}\ \Omega.
\end{align}
By definition \eqref{prob_r} of $r$ we have that $r=0$ on $\G{u}$ and $\nn\cdot\grad r=0$ on $\G{f}$.
Therefore, since $\bq\cdot\nn=0$ on $\G{f}$ and $\bq-\grad r=\curl p$, it follows that
$\nn\cdot\curl p=0$ on $\G{f}$, $p|_\G{f}$ is constant (we select $0$). Furthermore,
since $\bt\cdot\grad r=\bt\cdot\grad u=0$ on $\G{u}$ and $\grad u-t^2\grad r=t^2\curl p+\bpsi$,
we obtain the following boundary conditions for $p$:
\begin{align*}
   p=0\quad\text{on}\quad\G{f},\quad
   & \bt\cdot(t^2\curl p+\bpsi)=0\quad\text{on}\quad \G{u}.
\end{align*}
Now, introducing $\eeta:=t\curl p$, substituting $\bq$ by decomposition \eqref{dec},
system \eqref{RM3} becomes

\begin{subequations} \label{prob}
\begin{alignat}{2}
    \Div\MM - \curl p &= \grad r  \label{p1},\\
    \MM + \cC\sGrad\bpsi &= 0     \label{p2},\\
    \rot(t\eeta+\bpsi) &= 0       \label{p3},\\
    t\curl p-\eeta    &= 0        \label{p4}
\end{alignat}
\end{subequations}
in $\Omega$, and $u$ is determined by
\begin{align} \label{prob_u}
   u\in H^1_u(\Omega):\quad
   \vdual{\grad u}{\grad\deltau} = t^2\vdual{f}{\deltau} + \vdual{\bpsi}{\grad\deltau}
   \quad\forall\deltau\in H^1_u(\Omega).
\end{align}
In the following we derive an ultraweak formulation for problem \eqref{prob},
considering different boundary conditions.

\section{Spaces and trace operators} \label{sec_spaces}

In the following let $\cT=\{T\}$ denote a mesh of pairwise disjoint Lipschitz elements $T$ covering
$\Omega$, $\bar\Omega=\cup_{T\in\cT}\bar T$.
For a Lipschitz domain $\omega\subset\Omega$ we use the Lebesgue spaces of scalar,
vector and $2\times 2$ tensor fields
$L_2(\omega)$, $\bL_2(\omega)$ and $\LL_2(\omega)$, with generic norm $\|\cdot\|_\omega$,
denote by $\LL_2^s(\omega)$ the space of symmetric tensors, and define the Sobolev spaces
\[
   \Hrot{\omega}:=\{\brho\in\bL_2(\omega);\; \rot\brho\in L_2(\omega)\},
\]
\[
   \HDiv{\omega}:=\{\SS\in\LL_2(\omega);\; \Div\SS\in\bL_2(\omega)\},\quad
   \HDivs{\omega}:=\HDiv{\omega}\cap\LL_2^s(\omega).
\]
Corresponding product spaces with respect to $\cT$ are denoted analogously, replacing $\omega$
with $\cT$. For instance, $\Hrot{\cT}:=\Pi_{T\in\cT} \Hrot{T}$ with product norm
\[
   \bigl(\|\brho\|^2 + \|\rot\brho\|_\cT^2\bigr)^{1/2}\quad\text{where}\quad
   \|\rot\brho\|_\cT^2 := \sum_{T\in\cT} \|\rot\brho\|_T^2\quad \bigl(\brho\in\Hrot{\cT}\bigr).
\]
Furthermore, we need the spaces
\begin{subequations} \label{V}
\begin{align}
   &V_1(\cT) := \bH^1(\cT)\times \Hrot{\cT},\\
   &V_2(\cT) := \HDiv{\cT}\times H^1_*(\cT),\quad
    V_2^s(\cT) := \HDivs{\cT}\times H^1_*(\cT),\\
   &V(\cT) := V_1(\cT)\times V_2^s(\cT).
\end{align}
\end{subequations}
with $H^1_*(\cT):=H^1(\cT)/\R$ when $\G{f}=\emptyset$ and $H^1_*(\cT):=H^1(\cT)$ otherwise, and norms
\begin{align*}
   \|(\bchi,\brho)\|_{V_1(\cT,t)}^2 &:=
   \|\bchi\|^2 + \|\Grad\bchi\|_\cT^2 + \|\brho\|^2 + t^{-2} \|\rot(t\brho+\bchi)\|_\cT^2,\\
   \|(\SS,v)\|_{V_2(\cT,t)}^2 &:=
   \|\SS\|^2 + t_*^2 \|v\|_*^2 + \|\Div\SS-\curl v\|_\cT^2 + t^2 \|\curl v\|_\cT^2,\\
   \|(\bchi,\brho,\SS,v)\|_{V(\cT,t)}^2 &:= \|(\bchi,\brho)\|_{V_1(\cT,t)}^2 + \|(\SS,v)\|_{V_2(\cT,t)}^2
\end{align*}
for $(\bchi,\brho)\in V_1(\cT)$ and $(\SS,v)\in V_2(\cT)$. For notation $t_*$ and $\|\cdot\|_*$,
recall definition \eqref{star}.

In our variational setting, $V(\cT)=V_1(\cT)\times V_2^s(\cT)$
will be the test space with symmetric tensors $\SS\in\HDivs{\cT}$
whereas the larger space $V_2(\cT)$ will be needed to define traces of $\MM\in\HDiv{\Omega}$
without symmetry condition. Specifically, our trace variables will come from the trial spaces
\begin{align} \label{U_cont}
   U_1 := \bH^1(\Omega)\times \Hrot{\Omega}\subset V_1(\cT),\quad
   U_2 := \HDiv{\Omega}\times H^1_*(\Omega)\subset V_2(\cT) 
\end{align}
with norms
\begin{align*}
   \|(\bpsi,\eeta)\|_{U_1(t)}^2 &:=
   \|\bpsi\|^2 + \|\sGrad\bpsi\|^2 + \|\eeta\|^2 + t_*^{-2} \|\rot(t\eeta+\bpsi)\|^2,\\
   \|(\MM,p)\|_{U_2(t)}^2 &:=
   \|\MM\|^2 + t^2 \|p\|_*^2 + \|\Div\MM-\curl p\|^2 + t^2 \|\curl p\|^2 
\end{align*}
for $(\bpsi,\eeta)\in U_1$ and $(\MM,p)\in U_2$. Note the subtle difference in the norms
$\|\Grad\bchi\|_\cT$ with full gradient on the test side and $\|\sGrad\bpsi\|$ with symmetric part
of the gradient on the trial side. Of course,
\[
   \|(\bpsi,\eeta)\|_{U_1(t)}^2
   \simeq
   \|\bpsi\|^2 + \|\Grad\bpsi\|^2 + \|\eeta\|^2 + t^2 \|\rot\eeta\|^2
   \quad\text{if}\quad t_*=1
\]
uniformly in $t>0$.
In \eqref{U_cont} we have identified functions of $\cT$-product spaces
with their counterparts on $\Omega$ defined in the piecewise sense, and will continue to do so.
Taking the boundary conditions into account (recall definition \eqref{H1_psi_p}) the spaces become
\begin{subequations} \label{U0}
\begin{align}
   \label{U01}
   U_{1,0}(t) &:= \{(\bpsi,\eeta)\in \bH^1_{\psi}(\Omega)\times \Hrot{\Omega};\;
                    \bt\cdot(t\eeta+\bpsi)|_\G{u}=0\},\\
   \label{U02}
   U_{2,0}(t) &:= \{(\MM,p)\in \HDiv{\Omega}\times H^1_p(\Omega);\;
                    \bt\cdot\MM\nn|_\G{sc}= \nn\cdot\MM\nn|_\G{hss}=0,\
                    \MM\nn|_{\G{sss}\cup\G{f}}=0 \}. 
\end{align}
\end{subequations}

\subsection{Traces} \label{sec_traces}

We define trace operators
\[
   \tracepsi{}:\;  V_1(\cT)\to V_2^s(\cT)',\qquad
   \traceM{}:\;  V_2(\cT)\to V_1(\cT)'
\]
by
\begin{align*}
   \dual{\tracepsi{}(\bchi,\brho)}{(\SS,v)}_\cS
   &:=
   \vdual{\bchi}{\Div\SS-\curl v}_\cT - \vdual{\brho}{t\curl v}_\cT
   +\vdual{\Grad\bchi}{\SS}_\cT + \vdual{\rot(t\brho+\bchi)}{v}_\cT,
   \\
   \dual{\traceM{}(\MM,p)}{(\bchi,\brho)}_\cS
   &:= \dual{\tracepsi{}(\bchi,\brho)}{(\MM,p)}_\cS
\end{align*}
for $(\bchi,\brho)\in V_1(\cT)$, $(\SS,v)\in V_2^s(\cT)$, and $(\MM,p)\in V_2(\cT)$.
Here, $\vdual{\cdot}{\cdot}_\cT$ denotes the generic $L_2(\cT)$-duality for scalar, vector and
tensor fields. That is, appearing differential operators are taken piecewise on $\cT$.
Of course, for symmetric $\SS$, $\vdual{\Grad\bchi}{\SS}_\cT=\vdual{\sGrad\bchi}{\SS}_\cT$ in the
definition of operator $\tracepsi{}$.

Now, the restriction of $\tracepsi{}$ and $\traceM{}$ to the corresponding spaces of continuous
functions gives rise to the trace spaces
\[
   \Hpsiz(\cS,t) := \tracepsi{}(U_{1,0}(t)),\quad
   \HMz(\cS,t) := \traceM{}(U_{2,0}(t)) 
\]
with norms
\begin{subequations} \label{norms_trace}
\begin{align}
   \|\tpsi\|_{\psi,t} &:= \inf\{\|(\bpsi,\eeta)\|_{U_1(t)};\; \tracepsi{}((\bpsi,\eeta)) = \tpsi\}
   &&\bigl(\tpsi\in \Hpsiz(\cS,t)\bigr),\\
   \|\tM\|_{M,t} &:= \inf\{\|(\MM,p)\|_{U_2(t)};\; \traceM{}((\MM,p)) = \tM\}
   &&\bigl(\tM\in \HMz(\cS,t)\bigr).
\end{align}
\end{subequations}
The following statements relate canonical trace norms with their counterparts by duality.

\begin{lemma} \label{la_norms}
\begin{align*}
   \|\tpsi\|_{\psi,t}
   &= \sup_{0\not=(\SS,v)\in V_2^s(\cT)} \frac {\dual{\tpsi}{(\SS,v)}_\cS}{\|(\SS,v)\|_{V_2(\cT,t)}}
   \quad \bigl(\tpsi\in \Hpsiz(\cS,t)\bigr),\\
   \|\tM\|_{M,t}
   &= \sup_{0\not=(\bchi,\brho)\in V_1(\cT)}
      \frac {\dual{\tM}{(\bchi,\brho)}_\cS}{\|(\bchi,\brho)\|_{V_1(\cT,t)}}
   \quad \bigl(\tM\in \HMz(\cS,t)\bigr). 
\end{align*}
Here, the dualities between the trace spaces and corresponding product test spaces
are defined in the canonical way, to be consistent with the definition of the trace operators.
\end{lemma}

\begin{proof}
These statements follow from standard arguments, see~\cite[Lemma~A.10]{CarstensenDG_16_BSF}
for the first setting, and \cite[Proof of Lemma~4]{FuehrerHS_20_UFR} for a framework that applies
to our variational formulation. In \cite[Proof of Proposition~3]{FuehrerHN_DMS}) this
framework is recalled more briefly.
\end{proof}

\begin{remark} \label{rem_traces}
We note that for sufficiently smooth functions $(\bpsi,\eeta), (\bchi,\brho) \in U_1$,
$(\MM,p), (\SS,v)\in U_2$ the traces reduce to a mixture of classical traces on $\G{}$ of scalar functions,
tangential traces of vector functions, and normal traces of tensors
(with $L_2(\G{})$-bilinear form $\dual{\cdot}{\cdot}_\G{}$):
\begin{align*}
   \dual{\tracepsi{}(\bpsi,\eeta)}{(\SS,v)}_\cS
   &=
   \dual{\bpsi}{\SS\nn+v\bt}_\G{} + t \dual{\eeta\cdot\bt}{v}_\G{}
   =
   \dual{\bpsi}{\SS\nn}_\G{} + \dual{(t\eeta+\bpsi)\cdot\bt}{v}_\G{},\\
   \dual{\traceM{}(\MM,p)}{(\bchi,\brho)}_\cS
   &=
   \dual{\MM\nn+p\bt}{\bchi}_\G{} + t \dual{p}{\brho\cdot\bt}_\G{}
   =
   \dual{\MM\nn}{\bchi}_\G{} + \dual{p}{(t\brho+\bchi)\cdot\bt}_\G{}.
\end{align*}
For general functions of the above spaces, these are dualities between Sobolev spaces of
orders $\pm 1/2$ on $\G{}$. Furthermore, for test functions of product spaces,
$(\SS,v)\in V_2^s(\cT)$ and $(\bchi,\brho)\in V_1(\cT)$, traces $\tracepsi{}(\bpsi,\eeta)$
and $\traceM{}(\MM,p)$ live on the skeleton $\cS=\{\partial T;\; T\in\cT\}$.
\end{remark}

\begin{lemma} \label{la_cont}
Let $(\bchi,\brho,\SS,v)\in V(\cT)$ and $t>0$ be given. The equivalences
\begin{align*}
   (\bchi,\brho)\in U_{1,0}(t) &\quad\Leftrightarrow\quad
   \dual{\traceM{}(\MM,p)}{(\bchi,\brho)}_\cS = 0\quad \forall (\MM,p)\in U_{2,0}(t),\\
   (\SS,v)\in U_{2,0}(t) &\quad\Leftrightarrow\quad
   \dual{\tracepsi{}(\bpsi,\eeta)}{(\SS,v)}_\cS = 0\quad \forall (\bpsi,\eeta)\in U_{1,0}(t)
\end{align*}
hold true. 
\end{lemma}

\begin{proof}
In both cases, direction ``$\Rightarrow$'' can be seen by following Remark~\ref{rem_traces}
and observing that all the dualities on $\Gamma$ vanish due to the boundary conditions
of spaces $U_{1,0}(t)$ and $U_{2,0}(t)$. To show direction ``$\Leftarrow$'' in the first statement,
let $(\bchi,\brho)\in V_1(\cT)$ be given with
$\dual{\traceM{}(\MM,p)}{(\bchi,\brho)}_\cS = 0$ for any $(\MM,p)\in U_{2,0}(t)$.
First, selecting $p=0$, this gives
\[
  0 = \vdual{\bchi}{\Div\MM} +\vdual{\Grad\bchi}{\MM}_\cT\quad\forall \MM\in\HDiv{\Omega}
  \quad \text{``plus boundary condition''},
\]
that is, $\bchi\in\bH^1(\Omega)$. Second, selecting $\MM=0$, we conclude that
\[
   0 =
   \vdual{\bchi}{-\curl p} - \vdual{\brho}{t\curl p} + \vdual{\rot(t\brho+\bchi)}{p}_\cT
   \quad\forall p\in H^1_*(\Omega)\quad \text{``plus boundary condition''},
\]
that is, $t\brho+\bchi\in\Hrot{\Omega}$. Therefore, $(\bchi,\brho)\in U_1$. The boundary
conditions $(\bchi,\brho)\in U_{1,0}(t)$ follow from the corresponding conditions
in $U_{2,0}(t)$ and the dualities observed in Remark~\ref{rem_traces}.

The proof of ``$\Leftarrow$'' in the second statement is analogous. Specifically,
selecting $(\bpsi,\eeta)=(0,\eeta)$ yields $v\in H^1(\Omega)$, and
selecting $(\bpsi,\eeta)=(\bpsi,0)$ shows that $\SS\in\HDiv{\Omega}$. The symmetry
$\SS\in\HDivs{\Omega}$ is due to the imposed symmetry in $V_2^s(\cT)$, cf.~\eqref{V}.
Therefore, $(\SS,v)\in U_2$, and the boundary conditions $(\SS,v)\in U_{2,0}(t)$ can be seen as before.
%
\end{proof}

\section{Three-stage variational formulation and discretization} \label{sec_DPG}

We obtain a variational formulation of problem \eqref{prob} by testing, respectively,
\eqref{p1}, \eqref{p2}, \eqref{p3}, \eqref{p4} with $-\bchi$, $\cCinv\SS$, $v$, $\brho$,
and using trace operators $\tracepsi{}$, $\traceM{}$. This gives
\begin{align*}
   b(&(\bpsi,\eeta,\MM,p,\tpsi,\tM),(\bchi,\brho,\SS,v))
   :=\\
   &\vdual{\bpsi}{\curl v-\Div\SS}_\cT + \vdual{\MM}{\cCinv\SS+\sGrad\bchi}_\cT
   + \vdual{\eeta}{t\curl v-\brho}_\cT
   + \vdual{p}{\rot(t\brho+\bchi)}_\cT\\
   &+ \dual{\tpsi}{(\SS,v)}_\cS - \dual{\tM}{(\bchi,\brho)}_\cS
   =
   -\vdual{\grad r}{\bchi} 
\end{align*}
with $\tpsi=\tracepsi{}(\bpsi,\eeta)$ and $\tM=\traceM{}(\MM,p)$.
A solution $\uu=(\bpsi,\eeta,\MM,p,\tpsi,\tM)$ will be sought in space $U(t)$ defined as
\begin{align} \label{U}
   U(t) := \bL_2(\Omega) \times \bL_2(\Omega) \times \LL_2^s(\Omega) \times L_2^*(\Omega)
        \times  \Hpsiz(\cS,t) \times \HMz(\cS,t)
\end{align}
with (squared) norm
\[
   \|(\bpsi,\eeta,\MM,p,\tpsi,\tM)\|_{U(t)}^2 :=
   \|\bpsi\|^2 + \|\eeta\|^2 + \|\MM\|^2 + t^2 \|p\|_*^2 + \|\tpsi\|_{\psi,t}^2 + \|\tM\|_{M,t}^2.
\]
Finally, our three-stage variational formulation of the scaled Reissner--Mindlin problem
\eqref{RM3} consists in solving \eqref{prob_r}, then the ultraweak formulation of \eqref{prob},
and afterwards  problem \eqref{prob_u}, that is,
\begin{subequations} \label{VF}
\begin{alignat}{5}
   &r\in H^1_u(\Omega):\quad
   && \vdual{\grad r}{\grad\deltar} = \vdual{f}{\deltar} \quad \forall\deltar\in H^1_u(\Omega),
   \label{S1}\\
   &\uu=(\bpsi,\eeta,\MM,p,\tpsi,\tM)\in U(t):\quad
   && b(\uu,\vv) = -\vdual{\grad r}{\bchi} \quad \forall \vv=(\bchi,\brho,\SS,v)\in V(\cT),
   \label{S2}\\
   &u\in H^1_u(\Omega):\quad
   && \vdual{\grad u}{\grad\deltau} = t^2\vdual{f}{\deltau} + \vdual{\bpsi}{\grad\deltau}
   \quad \forall\deltau\in H^1_u(\Omega).
   \label{S3}
\end{alignat}
\end{subequations}
Problems \eqref{S1}, \eqref{S3} are obviously well posed and we are left with proving the
well-posedness of \eqref{S2}.

\begin{theorem} \label{thm_stab}
For given $t\in (0,1]$ and $r\in H^1_u(\Omega)$, there is a unique solution
$\uu=(\bpsi,\eeta,\MM,p,\tpsi,\tM)\in U(t)$ to \eqref{S2}. It is uniformly bounded:
\[
   \|\uu\|_{U(t)} \lesssim  \|\grad r\|
\]
with a hidden constant that is independent of $r$, $\cT$, and $t\in (0,1]$.
Furthermore, $(\bpsi,\eeta,\MM,p)\in U_{1,0}(t)\times U_{2,0}(t)$ solves \eqref{prob},
and $\tpsi=\tracepsi{}(\bpsi,\eeta)$, $\tM=\traceM{}(\MM,p)$.
\end{theorem}

A proof of this theorem is given in \S\ref{sec_pf}.
Theorem~\ref{thm_stab} immediately gives the well-posedness of \eqref{VF}.

\begin{cor} \label{cor_stab}
For given $t\in (0,1]$ and $f\in L_2(\Omega)$, there is a unique solution
$(r,\uu,u)\in H^1_u(\Omega)\times U(t)\times H^1_u(\Omega)$ to \eqref{VF}. It is uniformly bounded:
\[
   \|\grad r\|^2 + \|\uu\|_{U(t)}^2 + \|\grad u\|^2 \lesssim  \|f\|^2
\]
with a hidden constant that is independent of $f$, $\cT$, and $t\in (0,1]$.
\end{cor}

In order to discretize the three-stage formulation we use standard finite elements for
problems \eqref{S1} and \eqref{S3}, and a Petrov--Galerkin discretization of \eqref{S2} with
optimal test functions, the DPG method. For simplicity we use the same mesh $\cT$
for the three stages and consider lowest-order approximations.
Therefore, we now assume that $\cT$ is a regular triangular mesh.
In the following, $P^k(\cT)$ denotes the space of piecewise polynomials of degree $k\ge 0$.
Below, we will also use the corresponding vector and tensor-valued spaces $P^k(\cT)^2$ and
$P^k(\cT)^{2\times 2}$, respectively. For Poisson problems \eqref{S1}, \eqref{S3}
we use finite element space
$P^{1,c}_u(\cT) := P^1(\cT)\cap H^1_u(\Omega)$ consisting of continuous, piecewise linear polynomials.
Then, selecting an approximation space $U_h(t)\subset U(t)$  (also of lowest order),
a discrete formulation of \eqref{VF} is
\begin{subequations} \label{VFh}
\begin{align}
   &r_h\in P^{1,c}_u(\cT):\quad
   \vdual{\grad r_h}{\grad\deltar} = \vdual{f}{\deltar} \quad \forall\deltar\in P^{1,c}_u(\cT),
   \label{S1h}\\
   &\uu_h=(\bpsi_h,\eeta_h,\MM_h,p_h,\tpsi_h,\tM_h)\in U_h(t):\quad
   b(\uu_h,\ttt\deltauu) = -\vdual{\grad r_h}{\ttt^\chi\deltauu}
   \quad \forall \deltauu\in U_h(t),
   \label{S2h}\\
   &u_h\in P^{1,c}_u(\cT):\quad
   \vdual{\grad u_h}{\grad\deltau}
   = t^2\vdual{f}{\deltau}
   + \vdual{\bpsi_h}{\grad\deltau}
   \quad \forall\deltau\in P^{1,c}_u(\cT).
   \label{S3h}
\end{align}
\end{subequations}
Here, $\ttt:\;U(t)\to V(\cT)$ denotes the \emph{trial-to-test operator} defined by
\begin{align} \label{ttt}
   \ip{\ttt(\uu)}{\vv}_{V(\cT,t)} = b(\uu,\vv)\quad\forall\vv\in V(\cT)
\end{align}
with inner product $\ip{\cdot}{\cdot}_{V(\cT,t)}$ in $V(\cT)$ that induces norm $\|\cdot\|_{V(\cT,t)}$
and $\ttt^\chi\uu:=\bchi$ for $\ttt\uu=(\bchi,\brho,\SS,v)$.

The following theorem states the robust quasi-optimal best approximation of scheme \eqref{VFh}.

\begin{theorem} \label{thm_DPG}
For given $t\in (0,1]$ and $f\in L_2(\Omega)$,
there is a unique solution $(r_h,\uu_h,u_h)$ to \eqref{VFh}. It satisfies
\begin{equation} \label{Cea1}
   \|\grad (r-r_h)\| + \|\uu-\uu_h\|_{U(t)}
   \lesssim
   \|\grad (r-\tilde r_h)\| + \|\uu-\tilde\uu_h\|_{U(t)}
\end{equation}
and
\begin{equation} \label{Cea2}
   \|\grad (u-u_h)\| \lesssim
   \|\grad (r-\tilde r_h)\| + \|\uu-\tilde\uu_h\|_{U(t)} + \|\grad(u-\tilde u_h)\|
\end{equation}
for any $(\tilde r_h,\tilde\uu_h,\tilde u_h)\in P^{1,c}_u(\cT)\times U_h(t)\times P^{1,c}_u(\cT)$.
Here, $(r,\uu,u)$ is the solution of \eqref{VF}, and the hidden constant is independent of
$\cT$, $t\in (0,1]$, the discrete spaces and datum $f$.
\end{theorem}

In \S\ref{sec_pf} we give a proof of this theorem. Main ingredient is the stability
of the adjoint problem of \eqref{prob}, which is of the same type. This is the subject of
the following subsection.

\begin{remark} \label{rem_test}
(i) In practice, optimal test functions $\ttt\deltauu$ in \eqref{S2h} have to be approximated.
This is done by solving \eqref{ttt} in a finite-dimensional subspace of $V(\cT)$.
In order to prove the well-posedness and quasi-optimal convergence of the then fully discrete
scheme one has to show the existence of a corresponding Fortin operator,
cf.~\cite{GopalakrishnanQ_14_APD}.

(ii) Recall that, in the case of boundary conditions without free part ($\G{f}=\emptyset$),
both trial space $U(t)$ and test space $V(\cT)$ have a quotient space component.
In the case of $U(t)$, $p\in L_2(\Omega)/\R$ is unique only up to an additive constant, cf.~\eqref{U}.
In practice, this can be fixed by adding a rank-one term to the linear system to require
$\vdual{p}{1}=0$. In the case of the test space, component $v$ of $\vv=(\bchi,\brho,\SS,v)\in V(\cT)$
is taken in $H^1(\cT)/\R$, cf.~\eqref{V}.
This can be implemented by simply using (a discrete subspace of) $H^1(\cT)$ instead of the quotient
space since $b(\uu,(0,0,0,1))=0$ for any $\uu\in U(t)$.
\end{remark}

\subsection{Stability of the adjoint problem} \label{sec_adj}

In the following we denote $L_2^0(\Omega):=\{v\in L_2(\Omega);\; \vdual{v}{1}=0\}$.

The continuous adjoint problem is as follows.
\emph{For given $\bg_1\in\bL_2(\Omega)$, $\GG_2\in\LL_2^s(\Omega)$, $g_3\in L_2^0(\Omega)$
if $\G{f}=\emptyset$, $g_3\in L_2(\Omega)$ if $\G{f}\not=\emptyset$, and
$\bg_4\in\bL_2(\Omega)$ find $(\bchi,\brho)\in U_{1,0}(t)$ and $(\SS,v)\in U_{2,0}(t)$
with $\SS\in\LL_2^s(\Omega)$ such that}
\begin{subequations} \label{adj}
\begin{alignat}{2}
    -\Div\SS + \curl v &= \bg_1  \label{a1},\\
    \cCinv\SS + \sGrad\bchi &= \GG_2     \label{a2},\\
    \rot(t\brho+\bchi) &= g_3       \label{a3},\\
    t\curl v-\brho    &= \bg_4        \label{a4}.
\end{alignat}
\end{subequations}

\begin{prop} \label{prop_adj}
Problem \eqref{adj} has a unique solution $(\bchi,\brho)\in U_{1,0}(t)\subset V_1(\cT)$,
$(\SS,v)\in U_{2,0}(t)\subset V_2(\cT)$ with $\SS\in\LL_2^s(\Omega)$. It is bounded as
\[
   \|(\bchi,\brho,\SS,v)\|_{V(\cT,t)}^2 \lesssim
   \|\bg_1\|^2 + \|\GG_2\|^2 + t^{-2} \|g_3\|^2 + \|\bg_4\|^2
\]
with a generic constant that is independent of $t>0$, $\cT$, and the given data
$\bg_1,\GG_2,g_3,\bg_4$.
\end{prop}

\begin{proof}
We represent $\SS$ and $\brho$ via \eqref{a2} and \eqref{a4},
\begin{align} \label{Srho}
   \SS = \cC(\GG_2-\sGrad\bchi),\quad \brho = t\curl v-\bg_4.
\end{align}
Then, substituting $\SS$ in \eqref{a1} and testing with $-\deltachi$, and substituting
$\brho$ in \eqref{a3} and testing with $\deltav$, we obtain the following variational problem.
\emph{Find $(\bchi,v)\in \bH^1_\psi(\Omega)\times H^1_p(\Omega)$ such that}
\begin{subequations} \label{mixed}
\begin{align}
   \label{mixed_a}
   \vdual{\cC\sGrad\bchi}{\sGrad\deltachi} - \vdual{\curl v}{\deltachi}
   &= -\vdual{\bg_1}{\deltachi} + \vdual{\cC\GG_2}{\sGrad\deltachi},\\
   \vdual{\bchi}{\curl\deltav} + t^2\vdual{\curl v}{\curl\deltav}
   &= \vdual{g_3}{\deltav} + t \vdual{\bg_4}{\curl\deltav}
\end{align}
\emph{for any $(\deltachi,\deltav)\in \bH^1_\psi(\Omega)\times H^1_p(\Omega)$.}
Here we used that
\begin{align*}
   &\vdual{\Div\SS}{\deltachi} + \vdual{\SS}{\sGrad\deltachi} +
   \vdual{\rot(t\brho+\bchi)}{\deltav} - \vdual{t\brho+\bchi}{\curl\deltav}
   \\
   &= \dual{\SS\nn}{\deltachi}_\Gamma - \dual{(t\brho+\bchi)\cdot\bt}{\deltav}_\Gamma
   = 0
\end{align*}
\end{subequations}
due to the imposed boundary conditions, cf.~\eqref{U0}.
By coercivity of the bilinear form from \eqref{mixed},
Korn's and Poincar\'e's inequalities \eqref{Korn}, \eqref{Poincare},
there exists a unique solution to problem \eqref{mixed} with bound
\[
   \|\bchi\|^2 + \|\Grad\bchi\|^2 + t^2\|v\|_*^2 + t^2\|\curl v\|^2
   \lesssim
   \|\bg_1\|^2 + \|\GG_2\|^2 + t^{-2}\|g_3\|^2 + \|\bg_4\|^2.
\]
In the case that $\dual{v}{\deltachi\cdot\bt}_\Gamma=0$ for
$v\in H^1_p(\Omega)$, $\deltachi\in\bH^1_\psi(\Omega)$ (that is, when $\Gamma$ is only composed
of $\G{hc}$, $\G{hss}$ and $\G{f}$) then
$\vdual{\curl v}{\deltachi}=\vdual{v}{\rot\deltachi}$, and the surjectivity of
\[
   \rot:\;\begin{cases}
      \bH^1_\psi(\Omega) \to L_2^0(\Omega)
      & \quad\text{if } \G{f}=\emptyset,\\
      \bH^1_\psi(\Omega) \to L_2(\Omega)
      & \quad\text{otherwise},
   \end{cases}
\]
(see~\cite[Theorem~7.1]{ArnoldSV_88_RID}) allows to control $v$ more strongly due to the implied
inf-sup property
\[
   \sup_{0\not=\deltachi\in \bH^1_\psi(\Omega)}
   \frac {\vdual{\curl v}{\deltachi}}{\|\sGrad\deltachi\|} \gtrsim \|v\|_*.
\]
By \eqref{mixed_a} we conclude in those cases that
\[
   \|v\|_*^2 \lesssim \|\sGrad\bchi\|^2 + \|\bg_1\|^2 + \|\GG_2\|^2
             \lesssim \|\bg_1\|^2 + \|\GG_2\|^2 + t^{-2} \|g_3\|^2 + \|\bg_4\|^2.
\]
Finally, the bounds for $\|\brho\|$, $\|\SS\|$ are implied by \eqref{Srho} and the previous estimates,
and $\|\Div\SS-\curl v\|_\cT^2=\|\bg_1\|^2$,
$t^{-2} \|\rot(t\brho+\bchi)\|_\cT^2=t^{-2}\|g_3\|^2$ by \eqref{a1}, \eqref{a3}.
\end{proof}

The well-posedness of \eqref{adj} implies the following injectivity.

\begin{cor} \label{cor_inj}
Let $t>0$ be given. If $\vv\in V(\cT,t)$ satisfies $b(\deltauu,\vv)=0$ for any $\deltauu\in U(t)$ then
$\vv=0$.
\end{cor}

\begin{proof}
Let $\vv\in V(\cT,t)$ be given with $b(\deltauu,\vv)=0$ for any $\deltauu\in U(t)$.
Selecting $\deltauu=(\bpsi,\eeta,\MM,p,\tpsi,\tM)$ with arbitrary traces
$\tpsi\in\Hpsiz(\cS,t)$, $\tM\in \HMz(\cS,t)$ and field
variables $\bpsi$, $\eeta$, $\MM$, $p$ all zero, Lemma~\ref{la_cont} shows that
$\vv=(\bchi,\brho,\SS,v)\in U_{1,0}(t)\times U_{2,0}(t)$.
We conclude that $(\bchi,\brho,\SS,v)$ solves the adjoint problem \eqref{adj} with homogeneous
data. Therefore, $\vv=0$ by Proposition~\ref{prop_adj}.
\end{proof}

\subsection{Proofs of Theorems~\ref{thm_stab},~\ref{thm_DPG}} \label{sec_pf}

With the preparations made, proofs of our main theorems are standard. We recall the main steps.
Theorem~\ref{thm_stab} follows from the ingredients of the Babu\v{s}ka--Brezzi framework,
verified in the following.

\begin{enumerate}
\item {\bf Boundedness of the functional.}
\begin{align*}
   &-\vdual{\grad r}{\bchi} \le \|\grad r\| \|\bchi\| \le \|\grad r\| \|\vv\|_{V(\cT,t)}
   \quad\forall \vv=(\bchi,\brho.\SS,v)\in V(\cT)
\end{align*}
holds by definition of norm $\|\cdot\|_{V(\cT,t)}$.

\item {\bf Boundedness of the bilinear form.}  We make use of Lemma~\ref{la_norms}. Then
the bound $b(\uu,\vv)\lesssim \|\uu\|_{U(t)}\|\vv\|_{V(\cT,t)}$ for any
$\uu\in U(t)$ and $\vv\in V(\cT)$ holds by the Cauchy--Schwarz inequality and the definitions
of the norms.

\item {\bf Injectivity.}
\[
   \sup_{0\not=\deltauu\in U(t)} \frac {b(\deltauu,\vv)}{\|\deltauu\|_{U(t)}} > 0
   \quad\forall \vv\in V(\cT,t)\setminus\{0\}
\]
holds by Corollary~\ref{cor_inj}.

\item {\bf Inf-sup condition.} We follow the criteria given in \cite[Theorem~3.3]{CarstensenDG_16_BSF}:
inf-sup condition
\begin{align} \label{infsup}
   \sup_{0\not=\vv\in V(\cT)}
   \frac {b(\uu;\vv)}{\|\vv\|_{\VV(\cT,t)}}
   \gtrsim \|\uu\|_{U(t)} \quad\forall \uu=(\bpsi,\eeta,\MM,p,\tpsi,\tM)\in U(t)
\end{align}
follows from the inf-sup conditions
\begin{align}
   \label{infsup2}
   \sup_{0\not=\vv=(\bchi,\brho,\SS,v)\in V(\cT)}
   \frac{\dual{\tpsi}{(\SS,v)}_\cS - \dual{\tM}{(\bchi,\brho)}_\cS}{\|\vv\|_{V(\cT,t)}}
   \gtrsim
   \bigl(\|\tpsi\|_{\psi,t}^2 + \|\tM\|_{M,t}^2\bigr)^{1/2}
\end{align}
for any $(\tpsi,\tM)\in\Hpsiz(\cS,t)\times\HMz(\cS,t)$, and
\begin{align}
   \label{infsup1}
   \sup_{0\not=\vv=(\bchi,\brho,\SS,v)\in U_{1,0}(t)\times U_{2,0}(t)}
   &\frac{b(\bpsi,\eeta,\MM,p,0,0;\vv)}{\|\vv\|_{V(\cT,t)}}
   \gtrsim \bigl(
               \|\bpsi\|^2 + \|\eeta\|^2 + \|\MM\|^2 + t^2 \|p\|^2
            \bigr)^{1/2}
\end{align}
for any $(\bpsi,\eeta,\MM,p)\in \bL_2(\Omega)\times \bL_2(\Omega)\times \LL_2^s(\Omega)\times L_2(\Omega)$.

By Lemma~\ref{la_norms}, inf-sup property \eqref{infsup2} is satisfied with constant $1$.
In what follows, let $p_*$ denote the representant of
$p\in L_2^*(\Omega)=L_2(\Omega)/\R$ with $\vdual{p_*}{1}=0$ if $\G{f}=\emptyset$,
and $p_*:=p$ if $\G{f}\not=\emptyset$.
Inf-sup condition \eqref{infsup1} follows by application of Proposition~\ref{prop_adj}
with data
\(
   (\bg_1,\GG_2,g_3,\bg_4):= (\bpsi,\MM,t^2 p^*,\eeta)
\)
in problem \eqref{adj} with solution denoted as $\vv^*\in U_{1,0}(t)\times U_{2,0}(t)$:
\begin{align*}
   &\sup_{0\not=\vv\in U_{1,0}(t)\times U_{2,0}(t)}
   \frac{b(\bpsi,\eeta,\MM,p,0,0;\vv)}{\|\vv\|_{V(\cT,t)}}
   \ge
   \frac{\vdual{\bpsi}{\bg_1} + \vdual{\MM}{\GG_2} + \vdual{p}{g_3} + \vdual{\eeta}{\bg_4}}
        {\|\vv^*\|_{V(\cT,t)}}
   \\
   &\gtrsim
   \frac{\|\bpsi\|^2 + \|\MM\|^2 + t^2 \|p_*\|^2 + \|\eeta\|^2}
        {\|\bg_1\|^2 + \|\GG_2\|^2 + t^{-2} \|g_3\|^2 + \|\bg_4\|^2\bigr)^{1/2}}
   =
   \bigl(\|\bpsi\|^2 + \|\MM\|^2 + t^2 \|p\|_*^2 + \|\eeta\|^2\bigr)^{1/2}.
\end{align*}
\end{enumerate}
It is clear that $(\bpsi,\eeta,\MM,p)\in U_{1,0}(t)\times U_{2,0}(t)$ solves problem \eqref{prob},
and that $\tpsi=\tracepsi{}(\bpsi,\eeta)$, $\tM=\traceM{}(\MM,p)$.
Therefore, Theorem~\ref{thm_stab} is proved.

It remains to prove Theorem~\ref{thm_DPG}.
Let $\uu^h=(\bpsi^h,\eeta^h,\MM^h,p^h,\tpsi^h,\tM^h)$ denote
the solution of \eqref{S2} with datum $r_h$ instead of $r$.
We denote as $B:\;U(t)\to V(\cT)'$ the operator induced by bilinear form $b(\cdot,\cdot)$.
Interpretation of the DPG scheme as a minimum residual method and uniform
equivalence of the norms $\|B\cdot\|_{V(\cT,t)'}$ and $\|\cdot\|_{U(t)}$
(due to the uniform boundedness of $b(\cdot,\cdot)$ and inf-sup property \eqref{infsup}) show that
\begin{align*}
   \|\uu^h-\uu_h\|_{U(t)} \lesssim \|\uu^h-\tilde\uu_h\|_{U(t)}\quad\forall \tilde\uu_h\in U_h(t).
\end{align*}
By the same arguments, variational formulation \eqref{S2} is stable, that is,
\begin{align} \label{est2}
   \|\uu-\uu^h\|_{U(t)} \lesssim \|\grad(r-r_h)\|.
\end{align}
An application of the triangle inequality (twice) and quasi-optimal convergence of \eqref{S1h} prove that
\begin{align*}
   \|\grad r-\grad r_h\| + \|\uu-\uu_h\|_{U(t)} \lesssim
   \|\grad r-\grad\tilde r_h\| + \|\uu-\tilde\uu_h\|_{U(t)}
   \quad\forall\tilde r_h\in P^{1,c}_u(\cT),\;\tilde\uu_h\in U_h(t).
\end{align*}
This is \eqref{Cea1}.
We proceed in the same way to bound $\|\grad(u-u_h)\|$. Defining $u^h\in H^1_u(\Omega)$ by
\[
   \vdual{\grad u^h}{\grad\deltau}
   = t^2\vdual{f}{\deltau} + \vdual{\bpsi_h}{\grad\deltau}
   \quad \forall\deltau\in H^1_u(\Omega),
\]
standard estimates (stability of \eqref{S3} by \eqref{Poincare}, Galerkin orthogonality of \eqref{S3h},
triangle inequality) imply that
\begin{align} \label{est3}
   \|\grad(u^h-u_h)\|
   \lesssim \|\grad(u^h-\tilde u_h)\|
   &\le \|\grad(u-u^h)\| + \|\grad(u-\tilde u_h)\|
   \nonumber\\
   &\lesssim \|\bpsi-\bpsi_h\| + \|\grad(u-\tilde u_h)\|
   \quad\forall \tilde u_h\in P^{1,c}_u(\cT).
\end{align}
Using again the triangle inequality and bounding $\|\bpsi-\bpsi_h\|\le \|\uu-\uu_h\|_{U(t)}$,
an application of the previous estimate \eqref{Cea1} shows that \eqref{Cea2} holds.
This finishes the proof of Theorem~\ref{thm_DPG}.

\subsection{Robust DPG scheme for large domains} \label{sec_large}

Let us present the changes that are needed for our DPG scheme to be robust for larger domains,
cf.~Remark~\ref{rem_d}. We select a number $d=d(\Omega)>0$ so that the Korn and
Poincar\'e inequalities
\begin{align*}
   \|\bchi\|^2 + d^2 \|\Grad\bchi\|^2 &\lesssim d^2 \|\sGrad\bchi\|^2 \ \forall\bchi\in \bH^1_\psi(\Omega),
   \quad
   \|v\|_* \lesssim d \|\grad v\| \ \forall v\in H^1_p(\Omega)
\end{align*}
hold uniformly with respect to the diameter of $\Omega$. A standard choice is
$d=\diam(\Omega)$. We note that this tuning can be refined for an-isotropic domains,
cf.~\cite{FuehrerH_21_RDM}.
Analogously to the definition of $t_*$, cf.~\eqref{star}, we introduce
\begin{alignat*}{4}
   d_*:=1\quad \text{if}\ \G{sc}=\G{sss}=\emptyset
   \quad\text{and}\quad
   d_*:=d\quad \text{if}\ \G{sc}\cup\G{sss}\not=\emptyset.
\end{alignat*}
We then scale the norms as follows. In the test space we select
\begin{align*}
   \|(\bchi,\brho)\|_{V_1(\cT,t)}^2 &:=
   d^{-2} \|\bchi\|^2 + \|\Grad\bchi\|_\cT^2 + \|\brho\|^2 + t^{-2} d^2 \|\rot(t\brho+\bchi)\|_\cT^2
   &&\bigl((\bchi,\brho)\in V_1(\cT)\bigr),\\
   \|(\SS,v)\|_{V_2(\cT,t)}^2 &:=
   \|\SS\|^2 + t_*^2 d_*^{-2} \|v\|_*^2 + d^2 \|\Div\SS-\curl v\|_\cT^2 + t^2 \|\curl v\|_\cT^2
   &&\bigl((\SS,v)\in V_2(\cT)\bigr).
\end{align*}
For the trace norms we re-scale the norms in $U_1$ and $U_2$ (cf.~\eqref{U_cont}) as
\begin{align*}
   \|(\bpsi,\eeta)\|_{U_1(t)}^2 &:=
   d^{-2} \|\bpsi\|^2 + \|\sGrad\bpsi\|^2 + \|\eeta\|^2 + t_*^{-2} d_*^2 \|\rot(t\eeta+\bpsi)\|^2
   && \bigl((\bpsi,\eeta)\in U_1\bigr),\\
   \|(\MM,p)\|_{U_2(t)}^2 &:=
   \|\MM\|^2 + t^2 d^{-2} \|p\|_*^2 + d^2 \|\Div\MM-\curl p\|^2 + t^2 \|\curl p\|^2
   && \bigl((\MM,p)\in U_2\bigr),
\end{align*}
and define norms $\|\cdot\|_{\psi,t}$ and $\|\cdot\|_{M,t}$ in $\Hpsiz(\cS,t)$ and $\HMz(\cS,t)$,
respectively, as in \eqref{norms_trace}. Finally, the (squared) re-scaled norm in the trial space is
\[
   \|\uu\|_{U(t)}^2 :=
   d^{-2} \|\bpsi\|^2 + \|\eeta\|^2 + \|\MM\|^2 + t^2 d^{-2} \|p\|_*^2
   + \|\tpsi\|_{\psi,t}^2 + \|\tM\|_{M,t}^2
\]
for $\uu=(\bpsi,\eeta,\MM,p,\tpsi,\tM)\in U(t)$, cf.~\eqref{U}, where the trace norms are the
re-scaled ones.

All relations and estimates from
Section~\ref{sec_spaces} hold true uniformly with respect to the diameter
of $\Omega$, when replacing the norms with the re-scaled ones just defined.
In particular, the equality statements of Lemma~\ref{la_norms} are valid.
Furthermore, the stability statement from Proposition~\ref{prop_adj} becomes
\[
   \|(\bchi,\brho,\SS,v)\|_{V(\cT,t)}^2 \lesssim
   d^2 \|\bg_1\|^2 + \|\GG_2\|^2 + t^{-2} d^2 \|g_3\|^2 + \|\bg_4\|^2.
\]
Repeating the steps from \S\ref{sec_pf} then shows the equivalence of norms
\[
   \|\cdot\|_{U(t)}\simeq \|B\cdot\|_{V(\cT,t)'}\quad\text{in}\ U(t),
\]
uniformly with respect to $t\in (0,1]$ and $\diam(\Omega)$.

Finally, with respect to the quasi-optimal error estimates, stage 1 of scheme \eqref{VFh} stays
optimal in the $H^1(\Omega)$-seminorm. On the other hand, the estimate for stage 2 requires
stability \eqref{est2} which now renders
\[
   \|\uu-\uu^h\|_{U(t)}^2 \lesssim d^2 \|\grad(r-r_h)\|^2
\]
due to the weighting $d^{-1} \|\bpsi\|$ of the norm of test-function component $\bpsi$ in \eqref{S2h}.
We therefore have the a priori error estimates
\begin{align*}
   \|\grad (r-r_h)\|^2 &\lesssim \|\grad (r-\tilde r_h)\|^2
   &&\forall\tilde r_h\in P^{1,c}_u(\cT),\\
   \|\uu-\uu_h\|_{U(t)}^2 &\lesssim d^2 \|\grad (r-\tilde r_h)\|^2 + \|\uu-\tilde\uu_h\|_{U(t)}^2
   &&\forall\tilde r_h\in P^{1,c}_u(\cT),\;\tilde\uu_h\in U_h(t).
\end{align*}
Proceeding as in \eqref{est3}, we obtain the error estimate for the third stage,
\begin{align*}
   \|\grad(u-u_h)\|^2
   &\lesssim \|\bpsi-\bpsi_h\|^2 + \|\grad(u-\tilde u_h)\|^2
   \lesssim d^2 \|\uu-\uu_h\|_{U(t)}^2 + \|\grad(u-\tilde u_h)\|^2
   \\
   &\lesssim
   d^4 \|\grad (r-\tilde r_h)\|^2 + d^2\|\uu-\tilde\uu_h\|_{U(t)}^2 + \|\grad(u-\tilde u_h)\|^2
   \quad\forall \tilde r_h, \tilde u_h\in P^{1,c}_u(\cT), \tilde\uu_h\in U_h(t).
\end{align*}
These a priori error estimates for $r,\uu,u$ hold uniformly with respect to $t\in(0,1]$ and
the diameter of $\Omega$.



\section{Locking-free lowest-order discretization for hard-clamped plates} \label{sec_lf}

In order to prove our lowest-order scheme to be locking free, the solution must be sufficiently
regular. Furthermore, to be able to use standard discrete spaces we restrict our analysis to the
hard-clamped case. This is also the situation where regularity estimates are known.
In the following, $\|\cdot\|_m$ denotes the standard Sobolev norm in $H^m(\Omega)$
($m=1,2$) for scalar, vector and tensor functions. Our assumption is the following.

\begin{ass} \label{ass1}
We consider the hard-clamped case, $\G{}=\G{hc}$.
For given $f\in L_2(\Omega)$ and $t\in (0,1]$, let
$(\bpsi,p)\in \bH^1_0(\Omega)\times H^1(\Omega)$, $r, u\in H^1_0(\Omega)$
denote the solution components of problems \eqref{prob_r}, \eqref{prob}, \eqref{prob_u}.
We select $p$ with $\vdual{p}{1}=0$. The regularity estimate
\begin{align*}
    \|r\|_2 + \|u\|_2 + \|\bpsi\|_2 + \|p\|_1 + t\|p\|_2 \lesssim \|f\|
\end{align*}
holds true with a hidden constant that is independent of $f$ and $t$.
\end{ass}

For a special case, Arnold and Falk proved that this assumption is true. We note that this
includes the case where, on the right-hand side of \eqref{S2}, $r$ is replaced
with its conforming finite element approximation $r_h$.

\begin{prop} \label{prop_reg} \cite[Theorem~2.1]{ArnoldF_89_UAF}\\
If $\Omega$ is convex and $\cC$ is the identity tensor then Assumption~\ref{ass1} holds true.
\end{prop}

As previously specified in Section~\ref{sec_DPG}, we use regular triangular meshes which we now
assume to be shape regular, and denote $h:=\max\{\diam(T); T\in\cT\}$.
We approximate $H^1_0(\Omega)$ with
continuous, piecewise linear polynomials, the space being denoted as
$P^{1,c}_0(\cT)\subset H^1_0(\Omega)$, as before.
In the following we also need the piecewise polynomial space without boundary condition,
$P^{1,c}(\cT):=P^1(\cT)\cap H^1(\Omega)$.
We still have to specify the discrete subspace $U_h(t)\subset U(t)$.
To this end, let $\mathcal{RT}^0(\cT)\subset \Hdiv{\Omega}$ and
$\mathcal{ND}^0(\cT)\subset \Hrot{\Omega}$ denote the lowest-order
Raviart--Thomas and N\'ed\'elec spaces, respectively. Of course, $\mathcal{ND}^0(\cT)$ is a rotation
of $\mathcal{RT}^0(\cT)$.
Subspaces $\mathcal{ND}^0_0(\cT)\subset\Hrotz{\Omega}\subset\Hrot{\Omega}$
denote the corresponding spaces of functions with zero tangential component on $\G{}$.
We define the discrete trace spaces
\begin{align*}
   \Hpsih(\cS,t) &:= \tracepsi{}\Bigl(P^{1,c}_0(\cT)^2\times\mathcal{ND}^0_0(\cT)\Bigr)
   \subset \Hpsiz(\cS,t),\\
   \HMh(\cS,t) &:= \traceM{}\Bigl(\mathcal{RT}^0(\cT)^2\times P^{1,c}(\cT)\Bigr)
   \subset \HMz(\cS,t).
\end{align*}
Here, $\mathcal{RT}^0(\cT)^2\subset\HDiv{\Omega}$ indicates that the Raviart--Thomas elements
are taken row-wise. Recalling Remark~\ref{rem_traces}, we note that
$\Hpsih(\cS,t)$ consists of two components,
one of continuous, piecewise linear vector functions on $\cS$,
and the second of piecewise constant functions on $\cS$, plus homogeneous boundary conditions.
Space $\HMh(\cS,t)$ also has two components, the first amounting to piecewise constant
vector functions on $\cS$, and the second to continuous, piecewise linear polynomials on $\cS$.
The field variables are approximated by piecewise constant functions. Therefore, the discrete
approximation space for our DPG scheme \eqref{S2h} is
\begin{align} \label{Uh}
   U_h(t) := P^0(\cT)^2 \times P^0(\cT)^2 \times
             \Bigl(P^0(\cT)^{2\times 2}\cap\LL_2^s(\Omega)\Bigr) \times P^0(\cT)
   \times \Hpsih(\cS,t) \times \HMh(\cS,t).
\end{align}
Under Assumption~\ref{ass1} the resulting DPG scheme is locking free, as stated next.

\begin{theorem} \label{thm_DPG_lf}
Suppose that Assumption~\ref{ass1} holds true. In particular, we consider the hard-clamped
boundary condition. Furthermore, we assume that $\cC$ is a $C^1$-tensor.
For $f\in L_2(\Omega)$ and $t\in (0,1]$ let $\uu=(\bpsi,\eeta,\MM,p,\tpsi,\tM)\in U(t)$
and $u\in H^1_0(\Omega)$ denote the solutions of~\eqref{S2} and \eqref{S3}, respectively.
Let $\uu_h=(\bpsi_h,\eeta_h,\MM_h,p_h,\tpsi_h,\tM_h)\in U_h(t)$
and $u_h\in P^{1,c}_0(\cT)$ be their corresponding approximations from \eqref{S2h} and \eqref{S3h}.
Using regular, shape-regular triangular meshes, the error estimate
\begin{align*}
    \|\grad(u-u_h)\| + \|\bpsi-\bpsi_h\| + \|\MM-\MM_h\| \lesssim h \|f\|
\end{align*}
holds true with a hidden constant that is independent of $t$, $f$, and $\cT$.
\end{theorem}

\subsection{Proof of Theorem~\ref{thm_DPG_lf}}

We will employ some canonical approximation operators, recalled in the following together with
their approximation properties.

\begin{itemize}
\item
$\Pi_h^0:\; L_2(\Omega)\to P^0(\cT)$ is the $L_2(\Omega)$-orthogonal projection onto piecewise
constants with approximation property
\begin{align*}
   \|(1-\Pi_h^0)v\| \lesssim h\|v\|_1.
\end{align*}
We use the same notation $\Pi_h^0$ as component-wise application to vector or tensor functions. 

\item
$\Pi_h^\grad:\; L_2(\Omega) \to P^{1,c}_0(\cT))$ is a quasi-interpolation operator
(Scott--Zhang) which satisfies
\begin{align*}
   &\|\Pi_h^\grad v\| \lesssim \|v\|, \quad
   \|\Pi_h^\grad v\|_1 \lesssim \|v\|_1, \quad
   h^{-1}\|(1-\Pi_h^\grad) v\| + \|(1-\Pi_h^\grad) v\|_1 \lesssim h\|v\|_2.
\end{align*}
We use the same notation $\Pi_h^\grad$ as component-wise application to vector functions.

\item
$\Pi_h^\div:\;\Hdiv{\Omega}\to\mathcal{RT}^0(\cT)$ is the interpolation operator
introduced by Ern \emph{et al.} \cite[Section~3.1]{ErnGSV_ELG}. It satisfies
\begin{align*}
   &\|\Pi_h^\div\bphi\| \lesssim \|\bphi\| + h \|(1-\Pi_h^0)\div\bphi\|, \quad
   \|(1-\Pi_h^\div)\bphi\| \lesssim h \|\bphi\|_1, \\
   &\div(1-\Pi_h^\div)\bphi = (1-\Pi_h^0)\div\bphi.
\end{align*}
We use the same notation $\Pi_h^\div$ as row-wise application to tensor functions.

\item
$\Pi_h^{\rot}:\; \Hrotz\Omega\to\mathcal{ND}^0_0(\cT)$ is the rotated version of $\Pi_h^\div$,
mapping to the N\'ed\'elec space of rotated Raviart--Thomas elements, and including
the boundary condition of zero tangential components on $\G{}$.
Operator $\Pi_h^{\rot}$ has the same approximation properties as $\Pi_h^\div$,
replacing $\div$ with $\rot$.
\end{itemize}

Now, in order to prove Theorem~\ref{thm_DPG_lf}, it suffices to bound the right-hand side from
\eqref{Cea2} for appropriately selected discrete functions
$\tilde r_h,\tilde u_h\in P^{1,c}_0(\cT)$,
$\tilde\uu_h=(\bphi_h,\bxi_h,\NN_h,q_h,\tphi_h,\tN_h)\in U_h(t)$.

\begin{enumerate}
\item Regularity.
Recall the regularity properties from Assumption~\ref{ass1} where we include the replacement
of $r$ with $r_h$ in \eqref{p1} and \eqref{S2h}, and select $p$ with $\vdual{p}{1}=0$.
Since $\cC$ is $C^1$ by assumption, relations $\MM = -\cC\sGrad\bpsi$ and $\eeta=t\curl p$ imply that
\begin{align*}
    \|r\|_2 + \|u\|_2 + \|\bpsi\|_2 + \|p\|_1 + t\|p\|_2 + \|\MM\|_1 + \|\eeta\|_1 &\lesssim \|f\|.
  \end{align*}
This regularity will be used throughout in the following.

\item Approximation of $r$ and $u$.
We define $\tilde r_h:=\Pi_h^\grad r$ and $\tilde u_h=\Pi_h^\grad u$. This gives
\[
   \|\grad(r-\tilde r_h)\| + \|\grad(u-\tilde u_h)\| \lesssim h \|f\|.
\]

\item Approximation of the field variables from $\uu$. We choose
\begin{align*}
   \bphi_h := \Pi_h^0\bpsi, \quad \bxi_h := \Pi_h^0\eeta, \quad
   \NN_h := \Pi_h^0\MM, \quad q_h := \Pi_h^0p.
\end{align*}
The approximation properties of $\Pi_h^0$ show that
\begin{align*}
    \|\bpsi-\bphi_h\| + \|\eeta-\bxi_h\| + \|\MM-\NN_h\| + \|p-q_h\|
    &\lesssim 
    h\big( \|\bpsi\|_1 + \|\eeta\|_1 + \|\MM\|_1 + \|p\|_1  \big)
    \lesssim h\|f\|.
\end{align*}

\item Approximation of trace $\tpsi$. We select
\begin{align*}
   \tphi_h := \tracepsi{}(\Pi_h^\grad\bpsi,\Pi_h^{\rot}\eeta).
\end{align*}
By definition of trace operator $\tracepsi{}$ we find that, for $(\SS,v)\in V_2(\cT)$,
\begin{align*}
   \dual{\tpsi-\tphi_h}{(\SS,v)}_\cS
   &= \vdual{(1-\Pi_h^\grad)\bpsi}{\Div\SS-\curl v}_\cT - \vdual{(1-\Pi_h^{\rot})\eeta}{t\curl v}_\cT
   \\
   &\qquad+
   \vdual{\sGrad(1-\Pi_h^\grad)\bpsi}{\SS}_\cT +
   \vdual{\rot(t(1-\Pi_h^{\rot})\eeta+(1-\Pi_h^\grad)\bpsi)}{v}_\cT.
\end{align*}
Bounding the terms involving $(1-\Pi_h^\grad)\bpsi$ shows that
  \begin{align*}
    &|\vdual{(1-\Pi_h^\grad)\bpsi}{\Div\SS-\curl v}_\cT|
    + |\vdual{\sGrad(1-\Pi_h^\grad)\bpsi}{\SS}_\cT|
    + |\vdual{\rot(1-\Pi_h^\grad)\bpsi}{v}_\cT|
    \\
    &\qquad\lesssim
    \|(1-\Pi_h^\grad)\bpsi\|\,\|\Div\SS-\curl v\|
    + \|\sGrad(1-\Pi_h^\grad)\bpsi\|\,\|\SS\|
    + \|\rot(1-\Pi_h^\grad)\bpsi\|\,\|v\|_*
    \\
    &\qquad
    \lesssim h\|\bpsi\|_2 \|(\SS,v)\|_{V_2(\cT,t)}
    \lesssim h\|f\|\,\|(\SS,v)\|_{V_2(\cT,t)}.
\end{align*}
For the remaining terms we use the commutativity property
$\rot(1-\Pi_h^{\rot})\eeta = (1-\Pi_h^0)\rot\eeta$. This yields
\begin{align*}
   &|\vdual{(1-\Pi_h^{\rot})\eeta}{t\curl v}_\cT| + t\,|\vdual{\rot (1-\Pi_h^{\rot})\eeta}{v}_\cT| 
   \\
   &\qquad = |\vdual{(1-\Pi_h^{\rot})\eeta}{t\curl v}_\cT| + t\,|\vdual{(1-\Pi_h^0)\rot\eeta}{v}_\cT|
   \\
   &\qquad = |\vdual{(1-\Pi_h^{\rot})\eeta}{t\curl v}_\cT| + t\,|\vdual{\rot\eeta}{(1-\Pi_h^0)v}_\cT| 
   \\
   &\qquad \lesssim h\|\eeta\|_1\, t\,\|\curl v\|_{\cT} + h \|\eeta\|_1\, t\,\|\grad v\|_{\cT} 
   \lesssim
   h \|f\| \,\|(\SS,v)\|_{V_2(\cT,t)}.
\end{align*}
Here, we also used that $\|\grad v\|_\cT = \|\curl v\|_\cT$.

\item Approximation of trace $\tM$.
We select $q_h := \Pi_h^\grad p$, $\NN_h := \Pi_h^\div(\MM - R(p-q_h))$ with
\begin{align*}
   R = \begin{pmatrix} 0 & 1 \\ -1 & 0\end{pmatrix}
\end{align*}
so that $\Div R(\cdot) = \curl(\cdot)$.
Since $p-q_h\in H^1(\Omega)$ we have that $R(p-q_h)\in\HDiv\Omega$. Then we define
\begin{align*}
   \tN_h := \traceM{}(\NN_h,q_h).
\end{align*}
By definition of trace operator $\traceM{}$ we find that, for $(\bchi,\brho)\in V_1(\cT)$,
\begin{align*}
   \dual{\tM-\tN_h}{(\bchi,\brho)}_\cS
   &= \vdual{\Div(\MM-\NN_h)-\curl (p-q_h)}{\bchi}_\cT - \vdual{t\curl (p-q_h)}{\brho}_\cT
   \\
   &\quad +\vdual{\MM-\NN_h}{\sGrad\bchi}_\cT + \vdual{p-q_h}{\rot(t\brho+\bchi)}_\cT.
\end{align*}
We bound
\begin{align*}
   &|\vdual{t\curl (p-q_h)}{\brho}_\cT| + |\vdual{p-q_h}{\rot(t\brho+\bchi)}_\cT| 
   \\
   &\qquad\lesssim ht\|p\|_2 \|\brho\|
   + th^2\|p\|_2\, t^{-1}\|\rot(t\brho+\bchi)\|_{\cT}
   \lesssim h\|f\|\,\|(\bchi,\brho)\|_{V_1(\cT,t)}.
\end{align*}
The term $\vdual{\MM-\NN_h}{\sGrad\bchi}_\cT$ is straightforward to bound.
Using the properties of $\Pi_h^\div$, $\Pi_h^\grad$, and noting that
$(1-\Pi_h^0)\Div Rq_h = 0$ since $q_h$ is piecewise affine, we obtain
\begin{align*}
   |\vdual{\MM-\NN_h}{\sGrad\bchi}_\cT| &\lesssim \|\MM-\NN_h\|\,\|\sGrad\bchi\|_{\cT} 
   \lesssim (\|\MM-\Pi_h^\div\MM\| + \|\Pi_h^\div R(p-q_h)\|)\|(\bchi,\brho)\|_{V_1(\cT,t)}
   \\
   &\lesssim \big( h\|\MM\|_1 + \|p-q_h\| + h\|(1-\Pi_h^0)\Div R(p-q_h)\| \big)
             \|(\bchi,\brho)\|_{V_1(\cT,t)}
   \\
   &\lesssim h\big( \|\MM\|_1 + \|p\|_1 + \|(1-\Pi_h^0)\curl p\| \big)
             \|(\bchi,\brho)\|_{V_1(\cT,t)}
   \\
   &\leq h\big( \|\MM\|_1 + \|p\|_1 + \|\curl p\| \big)
         \|(\bchi,\brho)\|_{V_1(\cT,t)}
   \lesssim h\|f\|\,\|(\bchi,\brho)\|_{V_1(\cT,t)}.
\end{align*}
For the final term of the approximation of $\tM$ we note that, by construction of $\NN_h$,
\begin{align*}
   \Div(\MM-\NN_h)-\curl (p-q_h) &= \Div(1-\Pi_h^\div)\MM - \Div (1-\Pi_h^\div)R(p-q_h) 
   \\
   &= \Div(1-\Pi_h^\div)(\MM-R(p-q_h)) = (1-\Pi_h^0)\Div(\MM-R(p-q_h)) 
   \\
   &= (1-\Pi_h^0)\Div(\MM-R p) = (1-\Pi_h^0)(\Div\MM-\curl p).
\end{align*}
Recall from~\eqref{p1} that $\Div\MM-\curl p= \grad r$ with $r\in H_0^1(\Omega)$
being the solution to Poisson problem~\eqref{S1}. Therefore,
\begin{align*}
   |\vdual{\Div(\MM-\NN_h)-\curl (p-q_h)}{\bchi}_\cT|
   &= |\vdual{(1-\Pi_h^0)\grad r}{\bchi}| = |\vdual{(1-\Pi_h^0)\grad r}{(1-\Pi_h^0)\bchi}|
   \\
   &\lesssim h^2\|r\|_2 \|\bchi\|_1
   \lesssim h^2\|f\|\,\|(\bchi,\brho)\|_{V_1(\cT,t)}.
\end{align*}
Note that this term is of higher order.
\end{enumerate}

Collecting all the estimates this concludes the proof of Theorem~\ref{thm_DPG_lf}.

\section{Numerical experiments} \label{sec_num}
\def\errU{\|u-u_h\|_1}
\def\errPsi{\|\bpsi-\bpsi_h\|}
\def\errM{\|\MM-\MM_h\|}
\def\estTOT{\eta}
\def\estStageOne{\eta_1}
\def\estStageTwo{\eta_2}
\def\estStageThree{\eta_3}

Throughout we consider the scaled Reissner--Mindlin model \eqref{RM3} with identity tensor $\cC$.
We study four problems and their numerical solution by the three-stage scheme
\eqref{VFh} where discrete space $U_h$ is selected as in \eqref{Uh}, with boundary conditions
as needed. As noted in Remark~\ref{rem_test}(i), trial-to-test operator $\ttt$ from \eqref{ttt}
has to be approximated. For all the examples we do this by solving \eqref{ttt} in the piecewise
polynomial (degree $3$) subspace of $V(\cT)$, rather than in $V(\cT)$.
In three of the four cases, the exact solutions are known and we report on the approximation
errors $\|u-u_h\|_1$, $\|\bpsi-\bpsi_h\|$, and $\|\MM-\MM_h\|$
(recall that $\|\cdot\|_1$ refers to the standard $H^1(\Omega)$-norm).
Specifically, Problem 1 is that of a polynomial solution with hard-clamped boundary
and has been used in several publications before,
Problem 2 is derived from the Kirchhoff solution and has hard simple support,
for Problem 3 we consider a non-convex polygon with a combination of clamped and free boundary pieces,
and Problem 4 stems from Di Pietro and Droniou \cite{DiPietroD_DMR} and has a $t$-dependent regularity
(we use the corresponding non-homogeneous hard-clamped condition).

Problem 1 satisfies our conditions from Section~\ref{sec_lf} that guarantee that our
method is locking free.
We don't have an exact solution of Problem 3. In that case we plot the a posteriori error estimators.
The numerical results indicate the absence of locking for all the problems studied here.

\subsection{Example with polynomial solution}\label{sec:examplePolSol}

\begin{figure}
  \begin{center}
     \includegraphics[width=0.49\textwidth]{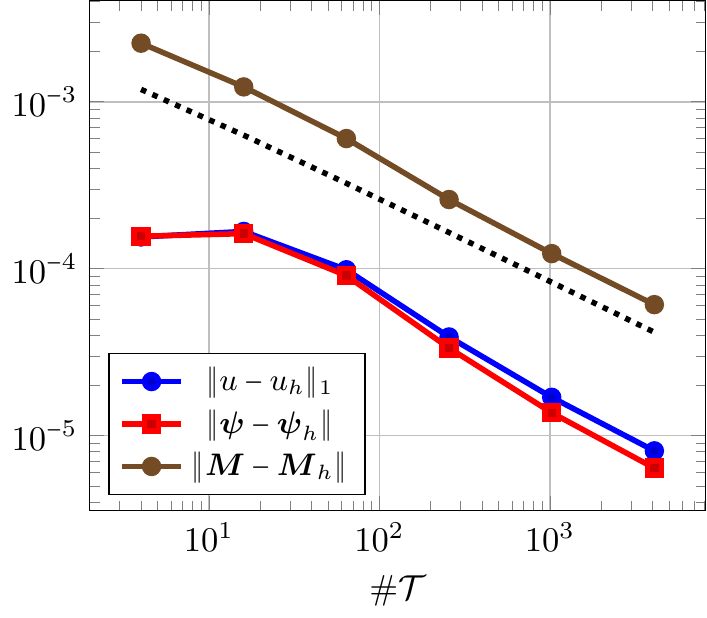}
     \includegraphics[width=0.49\textwidth]{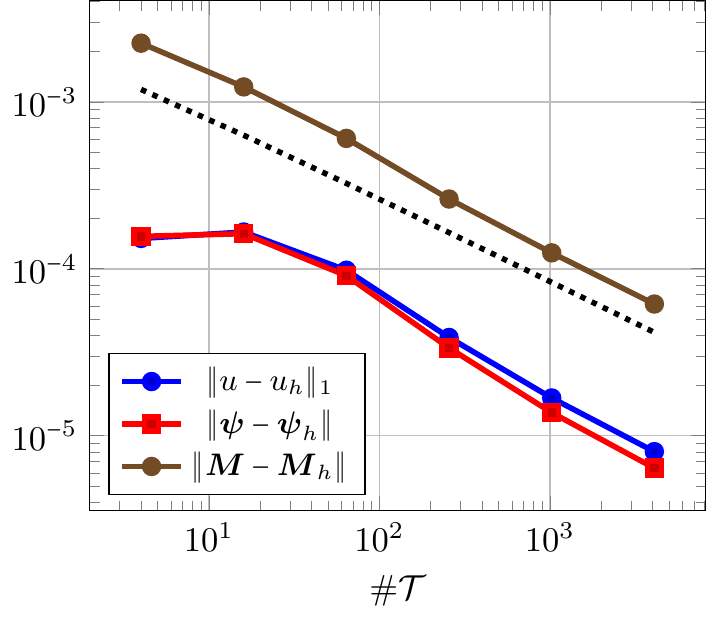}
  \end{center}
  \caption{Errors for the polynomial solution from Section~\ref{sec:examplePolSol}
           with $t=10^{-2}$ (left) and $t=10^{-4}$ (right).
           Dashed black lines indicate $\mathcal{O}(\#\cT^{-1/2})$.}
  \label{fig:polSol}
\end{figure}

We consider domain $\Omega = (0,1)^2$ and a manufactured polynomial solution. Starting with
\begin{align*}
  \bpsi(x,y) = \begin{pmatrix}
    y^3(y-1)^3x^2(x-1)^2(2x-1) \\
    x^3(x-1)^3y^2(y-1)^2(2y-1)
  \end{pmatrix},
\end{align*}
we calculate $\MM = -\sGrad\bpsi$, and $u\in H_0^1(\Omega)$ can be obtained from relation
\(
  \nabla u = t^2\div\MM + \bpsi.
\)
This solution satisfies the hard-clamped boundary condition. 

Figure~\ref{fig:polSol} shows the errors for plate thickness parameter
$t=10^{-2}$ (left plot) and $t=10^{-4}$ (right plot), using a sequence of uniformly refined meshes. 
As shown in Section~\ref{sec_lf}, our method is locking free in this case.
The numerical results confirm our result. 



\subsection{Example with hard simple support}\label{sec:exampleHSS}

\begin{figure}
  \begin{center}
     \includegraphics[width=0.49\textwidth]{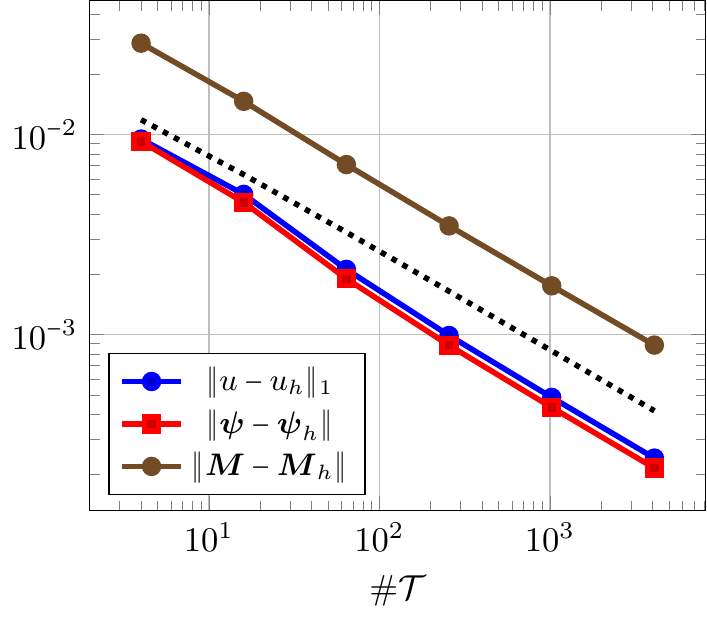}
     \includegraphics[width=0.49\textwidth]{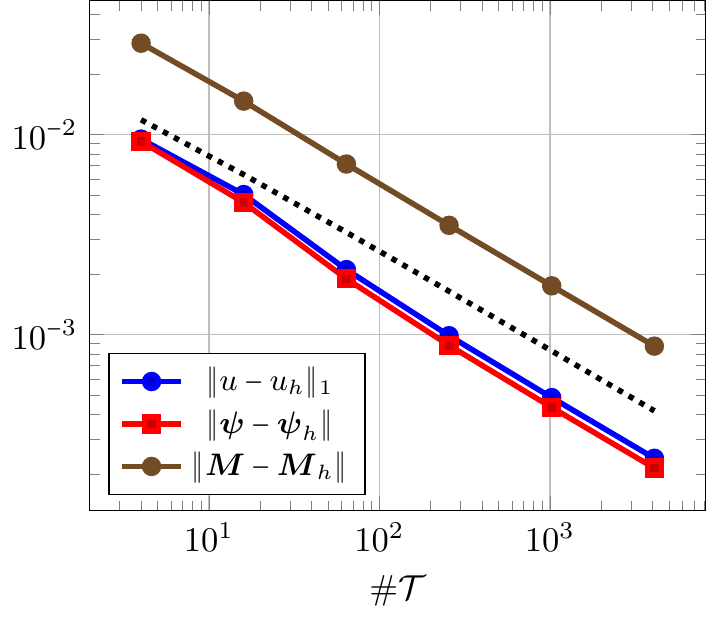}
  \end{center}
  \caption{Errors for example with Kirchhoff solution and hard simple support
           from Section~\ref{sec:exampleHSS},
           with $t=10^{-2}$ (left) and $t=10^{-4}$ (right).
           Dashed black lines indicate $\mathcal{O}(\#\cT^{-1/2})$.}
  \label{fig:HSS}
\end{figure}

In this example we select $\Omega = (0,1)^2$. Let $u_K$ be the Kirchhoff solution defined by
\begin{align*}
  \Delta^2 u_K &= 1 \quad\text{in }\Omega, \quad
  u_K = \Delta u_K =0 \quad\text{on }\Gamma.
\end{align*}
It can be represented as the Fourier series
\begin{align*}
  u_K(x,y) = \sum_{n,m=1}^\infty
             \frac{4 (1-\cos(m\pi))(1-\cos(n\pi))}{\pi^6 mn(m^2+n^2)^2}\sin(m\pi x)\sin(n\pi y).
\end{align*}
Defining
\begin{align*}
  u = u_K-t^2\Delta u_K, \quad \bpsi = \nabla u_K, \quad \MM = -\sGrad\bpsi, \quad \bq = -\nabla\Delta u_K
\end{align*}
one verifies that this gives a solution of \eqref{RM3}
with $\cC=\mathrm{id}$, $f = -\div\bq = 1$, and hard simple support.
To calculate approximation errors, we replace $u_K$ by its first $200^2$ Fourier terms. 

Figure~\ref{fig:HSS} shows the errors for $t=10^{-2}$ (left plot) and $t=10^{-4}$ (right plot),
using a sequence of uniformly refined meshes.
We observe that our results appear to be locking free also for this case with hard simple support.



\subsection{Example with singular solution}\label{sec:exampleLshape}

\begin{figure}
  \begin{center}
     \includegraphics[width=0.49\textwidth]{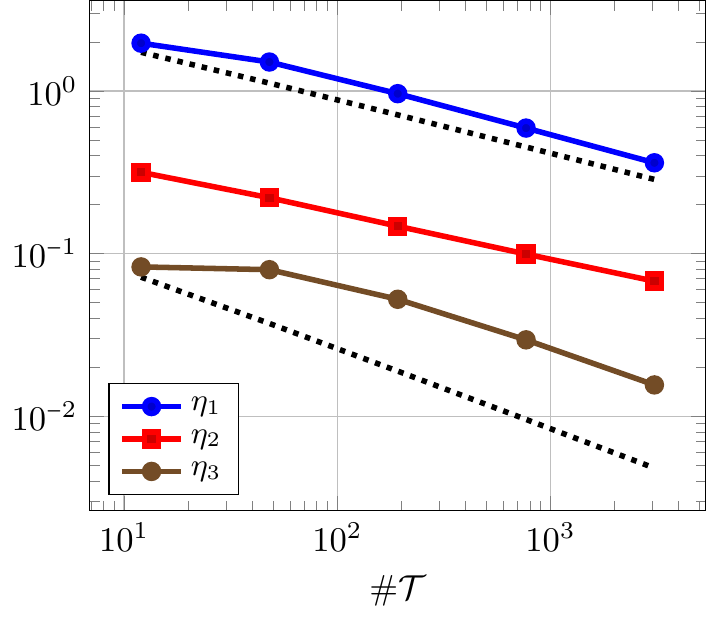}
     \includegraphics[width=0.49\textwidth]{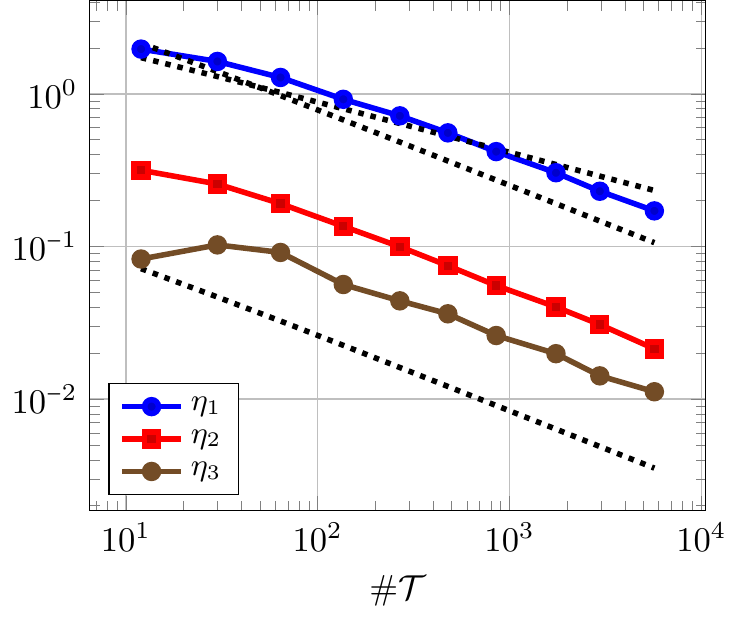}
  \end{center}
  \caption{Errors for example with singular solution from Section~\ref{sec:exampleHSS} and $t=10^{-3}$,
           uniform (left) and adaptively refined meshes (right).
           Dashed black lines indicate $\mathcal{O}(\#\cT^{-1/3})$ or $\mathcal{O}(\#\cT^{-1/2})$.}
  \label{fig:Lshape}
\end{figure}

\begin{figure}
  \begin{center} \includegraphics[width=0.7\textwidth]{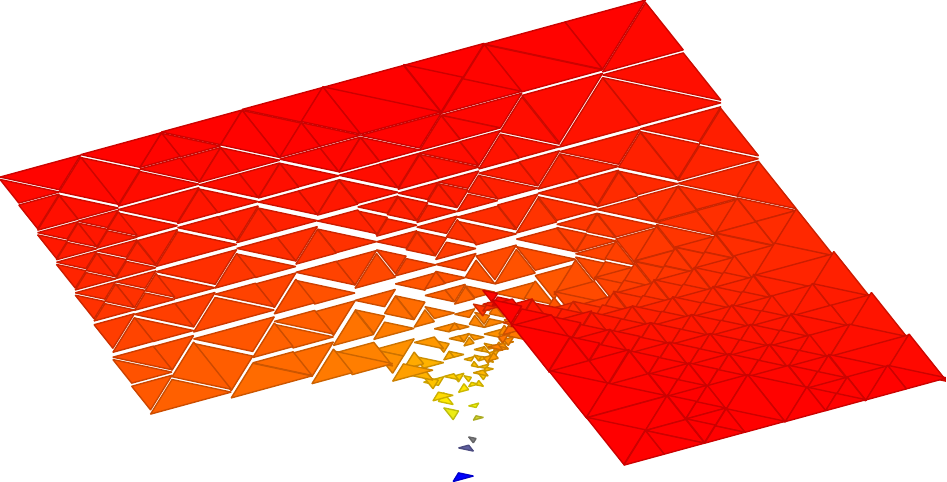} \end{center}
  \caption{Approximation of bending moment $\MM_{22}$ on adaptively refined mesh ($\#\cT = 477$)
           from the example with singular solution in Section~\ref{sec:exampleLshape}.}
  \label{fig:LshapeMoments}
\end{figure}

We consider the L-shaped domain $\Omega = (-1,1)^2\setminus[-1,0]^2$ and constant load $f=1$. 
The edges adjacent to the origin are hard clamped and all the others are free.
The exact solution to this problem is unknown, and due to the incoming corner we expect
the solution to have a corner singularity of reduced regularity.
Therefore, our method should exhibit a reduced order of convergence when using quasi-uniform meshes.
It is natural to consider adaptive mesh refinements to regain the optimal convergence rate.
To do so, we use a simple a posteriori error estimator composed of an estimator
for each of the three stages, written as
\[
  \estTOT^2 := \estStageOne^2+\estStageTwo^2+\estStageThree^2.
\]
Estimator $\estStageTwo$ denotes the built-in error estimator of the DPG method,
cf.~\cite{CarstensenDG_14_PEC} for details, and $\estStageOne$, $\estStageThree$
are weighted residual Poisson estimators, see, e.g.,~\cite[Chapter~2]{AinsworthO_00_AEE}.
All these estimators are localizable with respect to the elements of the mesh,
$\eta_j^2 = \sum_{T\in\cT}\eta_j(T)^2$, and thus can be used to steer a standard adaptive algorithm.
We use the bulk criterion
\begin{align*}
  \frac12 \eta^2 \leq \sum_{T\in\mathcal{M}} \eta^2(T)
\end{align*}
to mark elements and  newest-vertex bisection for mesh refinement.
Here, $\mathcal{M}$ denotes a (minimal) set of elements marked for refinement.

Figure~\ref{fig:Lshape} shows the three error estimators on a sequence of uniformly
(left plot) and adaptively refined meshes (right plot).
As expected, for uniform meshes we observe a reduced convergence rate whereas the optimal rate
is recovered through adaptivity. The presented results are for $t=10^{-3}$ and,
interestingly, we do not observe any locking effect. 


Figure~\ref{fig:LshapeMoments} shows component $\MM_{22}$ of the bending moment approximation
on an adaptively refined mesh. It seems to behave singularly at the incoming corner,
as do the other components (not shown).



\subsection{Example with $t$-dependent behavior}\label{sec:exampleDiPietroD}

\begin{figure}
  \begin{center}
     \includegraphics[width=0.49\textwidth]{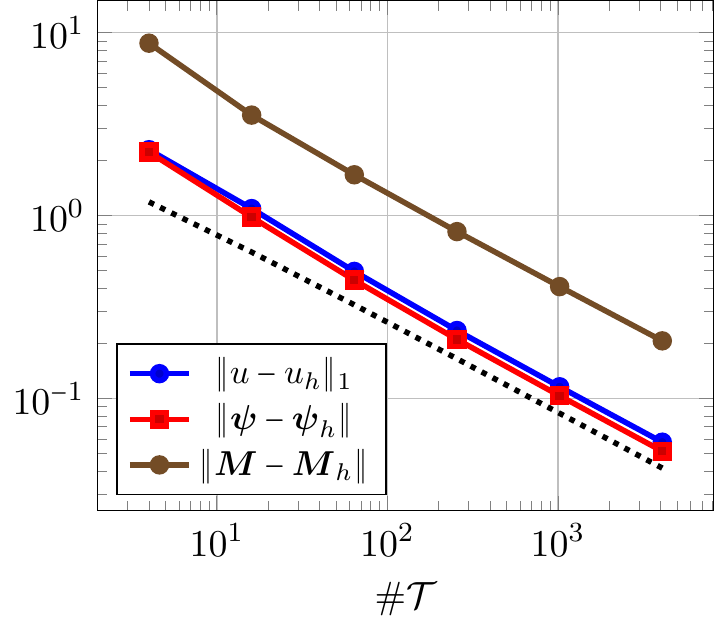}
     \includegraphics[width=0.49\textwidth]{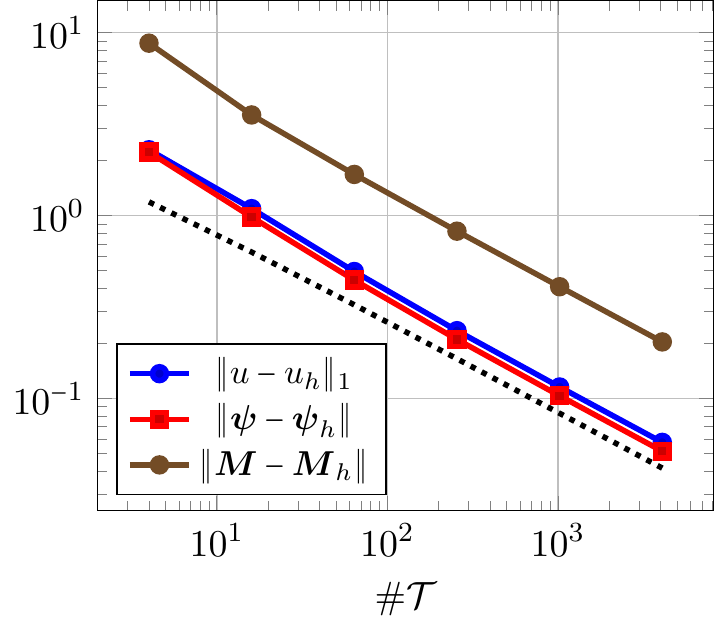}
  \end{center}
  \caption{Errors for example with $t$-dependent solution from Section~\ref{sec:exampleDiPietroD}
           with $t=10^{-2}$ (left) and $t=10^{-4}$ (right).
           Dashed black lines indicate $\mathcal{O}(\#\cT^{-1/2})$.}
  \label{fig:DiPietroD}
\end{figure}

We consider the manufactured solution from Di Pietro and Droniou~\cite{DiPietroD_DMR},
with domain $\Omega=(0,1)^2$.
We only note that in this case the $H^1(\Omega)$-norm of the shear force depends on $t$
and blows up like $t^{-1/2}$, whereas $\div\bq=-f$ is independent of $t$.
We refer to~\cite[Section~5.2.1]{DiPietroD_DMR} for details on the precise construction,
cf.~also \cite{ArnoldF_96_AAB}.
We impose hard-clamped boundary conditions and note that they are non-homogeneous. 

Figure~\ref{fig:DiPietroD} shows the errors for $t=10^{-2}$ (left plot) and $t=10^{-4}$ (right plot),
using a sequence of uniformly refined meshes. Again, we observe that our results are locking free.






\end{document}